\documentclass[review]{elsarticle}

\usepackage{lineno,hyperref}
\modulolinenumbers[5]

\journal{Journal of Computational Physics}

\usepackage{color}
\usepackage{graphicx}
\usepackage{ulem}
\usepackage{hyperref}
\usepackage{multirow}
\usepackage{amsmath,amstext,amssymb}
\usepackage{amsthm}
\usepackage{tabularx,longtable,multirow,diagbox}
\newcommand{\norm}[1]{\left\Vert#1\right\Vert}

\newcommand{\normE}[1]{\left\Vert\hskip -0.8pt \left\vert #1 \right\vert\hskip -0.8pt\right\Vert_{h,h}}

\newcommand{\pd}[1]{\left\langle #1\right\rangle}

\newcommand{\set}[1]{\left\{#1\right\}}

\newcommand{\M}{\mathcal{M}}

\newcommand{\C}{\mathcal{C}}
\newcommand{\CC}{(\mathcal{C}_{h,H}|_{V_{h,H}})^{-1}\mathcal{C}_{h,H}}

\newcommand{\nn}{\nonumber}
\newcommand{\ls}{\lesssim}

\newcommand{\al}{\alpha}

\newcommand{\ga}{\gamma}
\newcommand{\Ga}{\Gamma}

\newcommand{\na}{\nabla}
\newcommand{\om}{\omega}
\newcommand{\Om}{\Omega}
\newcommand{\pa}{\partial}

\newcommand{\dx}{\,\mathrm{d}x}

\newcommand{\eq}[1]{\begin{align}#1\end{align}}
\newcommand{\eqn}[1]{\begin{align*}#1\end{align*}}

\definecolor{gold}{rgb}{0,0,1}

\newtheorem{theorem}{Theorem}

\newtheorem{lemma}{Lemma}
\newtheorem{proposition}{Proposition}
\newtheorem{definition}{Definition}
\newtheorem{remark}{Remark}

\newtheorem{example}{Example}

\numberwithin{equation}{section} \numberwithin{theorem}{section}
\numberwithin{corollary}{section} \numberwithin{lemma}{section}

\begin{document}
\begin{frontmatter}

\title{A combined multiscale finite element method based on the LOD technique for the multiscale elliptic problems with singularities}
\author[nju]{Kuokuo Zhang\fnref{myfootnote}}
\ead{dg1721016@smail.nju.edu.cn}

\author[nju]{Weibing Deng\corref{mycorrespondingauthor}\fnref{myfootnote1}}
\ead{wbdeng@nju.edu.cn}
\cortext[mycorrespondingauthor]{Corresponding author}
\fntext[myfootnote1]{The work of this author was
partially supported by the NSF of China grant 12090023 and 12171237, and by the
Ministry of Science and Technology of China grant 2020YFA0713800.}

\author[nju]{Haijun Wu\fnref{myfootnote2}}
\ead{hjw@nju.edu.cn}
\fntext[myfootnote2]{The work of this author was
partially supported by the NSF of China grants 11071116, 91130004.}

\address[nju]{Department of Mathematics,
Nanjing University, Nanjing 210093, People's Republic of China}

\begin{abstract}
In this paper, we construct a combined multiscale finite element method (MsFEM) using the Local Orthogonal Decomposition (LOD) technique to solve the multiscale problems which may have singularities in some special portions of the computational domain. For example, in the simulation of steady flow transporting through highly heterogeneous porous media driven by extraction wells, the singularities lie in the near-well regions. The basic idea of the combined method is to utilize the traditional finite element method (FEM) directly on a fine mesh of the problematic part of the domain and using the LOD-based MsFEM on a coarse mesh of the other part. The key point is how to define local correctors for the basis functions of the elements near the coarse and fine mesh interface, which require meticulous treatment. The proposed method takes advantages of the traditional FEM and the LOD-based MsFEM, which uses much less DOFs than the standard FEM and may be more accurate than the LOD-based MsFEM for problems with singularities. The error analysis is carried out for highly varying coefficients, without any assumptions on scale separation or periodicity. {Numerical examples with periodic and random highly oscillating coefficients}, as well as the multiscale problems on the L-shaped domain, and multiscale problems with high-contrast channels or well-singularities are presented to demonstrate the efficiency and accuracy of the proposed method.
\end{abstract}


\end{frontmatter}

\linenumbers
\section{Introduction}
In this paper we consider the elliptic problems with rapidly varying (non-periodic) coefficients, which involve many spatial scales. Such problems are typically referred to as multiscale problems and often arisen in composite materials and flows in porous media. Any meaningful numerical simulation of these problems such as standard finite element method (FEM) has to account for the highly heterogeneous fine-scale structures in the whole computational domain. This means that the underlying computational mesh has to be sufficiently fine and hence requires an enormous computational demand.

In order to overcome this difficulty, many kinds of methods have been developed in recent decades to solve such multiscale problems. Roughly speaking, from the perspective of final approximation solution, these numerical methods can be categorized into two classes. One is to solve the original problem in the constructed coarse-grid multiscale basis function space hence obtains a good approximation of the original-problem solution; the other is to solve a macro model equivalent to the original problem on the coarse grid mesh hence grasps the macro behavior of the multiscale solution. See, for example, the generalized FEM (GFEM) \cite{BCO1994,MR2249154,MR1734667}, the multiscale FEM (MsFEM) \cite{MsFEM1997,MsFEM2000,CH2002}, the variational multiscale methods (VMM) or residual-free bubbles method (RFBM) \cite{VMM1995,VMM1998,BFHR,FY}, the heterogeneous multiscale methods (HMM) \cite{HMM2003,HMM2005,MR2916381}, the multiscale finite-volume method \cite{JLT2003}, the multigrid numerical homogenization techniques \cite{FB2, MR3605827}, the mortar multiscale methods \cite{MotMsFEM2002, MotMsFEM2007}, the localized orthogonal decomposition methods (LODM) \cite{MLocalization,HenningOversampling,HenningLocalized}, the equation-free approaches \cite{MR2341798,MR2306402}, the generalized MsFEM (GMsFEM) \cite{GMsFEM2013JCP, MR3506953}, the multiscale-spectral GFEM (MS-GEFM) \cite{MR4076862, MR2801210}, the constraint energy minimizing GMsFEM (CEM-GMsFEM)\cite{2018GMsFEM,MR3905735,MR3992043}, some numerical homogenization methods or upscaling methods \cite{OZh2007,OZh2011,up1991,up2000,up2002}, and so on.

Most of the above mentioned multiscale methods consist of two parts, one is the macro solver on coarse mesh such as various finite element or finite volume methods, and the other is the cell problems solving on the coarse grid or oversampling elements. The multiscale algorithm captures the fine-scale information of the solution by solving the cell problems, and then uses the solutions of the cell problems to form an equivalent macro model or a low dimensional multiscale approximation space of the solution. The definition of the cell problem is mainly based on the differential operator of the original multiscale problem, such as the elliptic operator, so that the variability of the multiscale coefficients can be brought into the final solution model through the solution of the cell problems.

In this paper, we are concerned with a special kind of multiscale problems -- those with singularities. For example, the one in L-region has singularity near the corner; while the problem with high-contrast channel that connects the boundaries of coarse-grid blocks has singularity at the edge of the channel \cite{EGW2011,GE20101,GE20102,OZh2011}; furthermore, the problem with steady flow transporting through highly heterogeneous porous media driven by extraction wells has singularity near the well \cite{CY2002}.
The traditional multiscale methods on coarse grids may be inefficient when dealing with singularity. This is mainly because the local singularity of the solution is hard to be grasped effectively at the coarse grid level. To solve this kind of singular problems, some numerical methods have been proposed in the literature. See, for instance,
the adaptive GMsFEM used to solve the high contrast problem \cite{Chung:Efendiev:Li:2014,MR3840879}, the MsFEM used to solve the high contrast interface and channel problems \cite{CGH2010,EGW2011}, the complete multiscale coarse grid algorithm by using the Green functions for solving steady flow problem involving well singularities in heterogeneous porous medium \cite{CY2002}, the CEM-GMsFEM used to solve the high contrast problem \cite{2018GMsFEM,MR3905735}, the LODM used to solve high contrast and complex geometric boundary problems \cite{HenningLocalized, MR3655818, MR3707891}, the combined MsFEM used to solve high contrast channel and well--singularity problems \cite{DW2014,DW2016}, and some generalized finite element methods and numerical homogenization methods used to solve high contrast problems \cite{MR3605827,MR4076862, MR2801210,OZh2007,OZh2011}, and so on. Among them, most of the multiscale methods capture the small scale information of the original-problem solution through the solution of the cell problems. Moreover, for the problems with singularities, such as the problem with high-contrast and narrow channels, in order to grasp the singularities, it needs to construct multiscale finite element approximation space via solving special cell problems. For example, the CEM-GMsFEM first needs to construct the auxiliary multiscale functions by solving the local spectral problem. Consequently, the auxiliary function space is constructed by selecting the eigenfunctions corresponding to small eigenvalues, which correspond to high contrast channels. Finally, the online multiscale basis functions are constructed based on constrained energy minimization in the auxiliary function space.

However, it is difficult to define the corresponding subproblems to construct the required approximation space for the problems with source term singularity, such as the porous medium flow problem with well singularities. In \cite{DW2014,DW2016}, the authors combined the standard FEM with the oversampling MsFEM and Petrov-Galerkin MsFEM to solve the multiscale problem with singularity.  The standard FEM is used on a fine mesh of the problematic part of the domain and the oversampling MsFEM or Petrov-Galerkin MsFEM is used on a coarse mesh of the other part. The transmission condition on the interface between coarse and fine meshes is dealt with the penalty technique.
The proposed methods take the advantages of the standard FEM and the MsFEM, and maintain the accuracy of the two methods. It is shown \cite{DW2014,DW2016} that the combined multiscale methods can solve the multiscale elliptic problems with fine and long-ranged high contrast channels and the well singularities very efficiently. But, the error analysis of the methods is still based on the classical homogenization theory, which requires the assumption that the diffusion coefficient is periodic. Therefore, how to improve the algorithm so that the optimal error estimate can be obtained for any diffusion coefficient needs further study. We remark that in the past decade, there are many nice multiscale methods dealing with arbitrary oscillating coefficients, such as the LODM, CEM-GMsFEM, MS-GFEM, and some numerical homogenization methods mentioned above \cite{MR3605827, MLocalization, HenningOversampling, HenningLocalized, MR4076862, MR2801210, 2018GMsFEM,MR3905735,MR3992043, OZh2007,OZh2011}.

In this paper, we focus on the LODM which was originally introduced in \cite{MLocalization} and could be derived from the VMM framework \cite{HenningOversampling,HenningLocalized}. The orthogonal decomposition method starts from two finite element spaces, a coarse space $V_H$ and a very high dimensional space $V_h$ which can approximate the multiscale solution well. Further, the decomposition can be described in three steps: (1) define a quasi-interpolation operator $I_H : V_h \rightarrow V_H$, (2) define a high dimensional space of negligible information
by the kernel of the operator $I_H$, i.e. $W_h$:= kern($I_H$), and (3) find the orthogonal complement of $W_h$
in $V_h$ with respect to the energy scalar product. With this strategy, it is possible to split the
space $V_h$ into the orthogonal direct sum of a low dimensional multiscale space $V_{H}^{ms}$
and a high dimensional remainder space $W_h$. The multiscale problem is solved in the
low dimensional space $V_{H}^{ms}$ and is therefore cheap. However, the construction of the
exact splitting of $V_h$ = $V_{H}^{ms}\oplus W_h$ is unpractical since it needs to define the correction operator in the whole domain which is computationally expensive. We call the method as an ideal one whose solution is referred as ideal solution.
To reduce the computational complexity, several localization strategies were proposed and analyzed in  \cite{MLocalization,HenningOversampling,HenningLocalized}. In fact, the computation of the orthogonal decomposition is localized to the patches of the elements, which we introduced as LODM. The reason which makes localization successful is that outside of the support of the coarse finite element basis functions of $V_H$, the canonical basis functions of the multiscale space $V_{H}^{ms}$ have the property of exponential decay. We remark that the LODM often use the fine-scale solution in $V_h$ for comparison, which is referred to as reference solution.

The essence of the LODM is to construct a low-dimensional solution space (with a locally supported basis functions) that has very accurate approximation properties with respect to the exact solution. So far, the idea of LOD has been generalized to several kinds of discretization techniques such as discontinuous Galerkin \cite{ElfversonConvergence}, Petrov-Galerkin formulations \cite{petrovlod} and mesh-free methods \cite{MR3587383}. Moreover, the method has been successfully applied to many kinds of problems such as semi-linear elliptic
problems \cite{HenningA}, eigenvalue problems \cite{MComputation, HenningTwo}, problems on complicated geometries \cite{MR3655818}, and so on. We refer the reader to \cite{MR3926249, MR3616023} and references therein for more works about LODM. The attractive point of this method is that it does not rely on the classical homogenization theory and does not need the scale separation assumption.

Based on the above observation, we will use the LOD technique to improve the combined MsFEM and make it suitable for general multiscale problems.
Note that the traditional FEM has many excellences to deal with the singularities, such as, refining the mesh or enlarging the polynomial order of the finite element space. Thus, in order to take advantages of both methods, we introduce a combined FE and LOD method (FE--LODM) to solve the multiscale problems with singularities. The idea of this approach is to utilize the traditional FEM directly on a fine mesh of the problematic part of the domain and use the LODM on a coarse mesh of the other part.
{Comparing to the implement of LODM, there are two key issues of the FE--LODM to consider. The first one is how to define the corresponding quasi-interpolation operator in the subdomain using fine mesh. Here we just choose the $L^2$ projection $\Pi_h$, which has the property that $\Pi_h u_h=u_h$ for $u_h$ belongs to the fine mesh linear FEM space. This property is very important in our later error analysis, which yields a very useful result that the ideal solution is equals to the reference solution in the subdomain using fine mesh.
The second one is how to define the correction operators near the interface between the coarse and fine mesh. A delicate treatment should be done for the elements who have an edge or face in the interface of coarse--fine mesh.
} For the introduced FE-LODM, we carry out a rigorous and careful analysis for the elliptic equation with arbitrary  diffusion coefficient to show both the energy and $L^{2}$ errors of the method have the {optimal} convergence rate.
The numerical results also show that the proposed FE-LODM is very efficient for multiscale problems with random generated coefficients and singularities.

The rest of this paper is organized as follows. In Section 2, we give the
model problem and define a fine-scale reference problem. Section 3 is devoted
to deriving the FE-LODM. In Section 4, we present the error analysis of the approach. In Section 5, we provide some numerical results to demonstrate the efficiency of our method.  Conclusions are draw in the last section.

Throughout this paper, standard notations for Lebesgue and Sobolev spaces are employed, and $C$ denotes the generic constant, which depends on neither the mesh size nor the diffusion coefficient. We also use the shorthand notation $a\lesssim b$ and $b\gtrsim a$ for the inequality $a \leq Cb$ and $b\geq Ca$.
In addition, the shorthand notation $a \eqsim b$ represents that $a\lesssim b$ and $b\lesssim a$.

\section{{Model problem and reference approximations}}
In this section, we first present the {multiscale} model problem, {then introduce its interior penalty continuous-discontinuous Galerkin (IPCDG) discretization on fine meshes and discuss the approximation errors of the IPCDG method. The IPCDG solution will be used as a fine-scale reference solution to estimate the error of the FE-LODM. Note that the IPCDG method was first introduced in \cite{cao_wu} for the Helmholtz equation.}

\subsection{Model Problem}
In this paper, we consider a second order elliptic problem with highly varying diffusion coefficient. Let $\Omega\subset\mathbb{R}^d$, $d$=2,3, be
a {polygonal/polyhedral}  domain, and the elliptic equation reads as
\begin{equation}\label{model}
 \left\{ \begin{aligned}
 &-\nabla\cdot(A\nabla u)=f\quad \mathrm{in}\,\,\Omega,\\
 & \qquad\qquad\quad u=0 \quad \mathrm{on}\,\,\partial\Omega,
\end{aligned} \right.
 \end{equation}
 where we assume that $f\in L^2(\Omega)$, and the diffusion
 matrix $A\in L^\infty(\Omega,{\mathbb{R}^{d\times d}})$
 is a symmetric {matrix} with uniform spectral
 bounds $\beta\geq\alpha>0$, {i.e.}
 \begin{equation}\begin{aligned}
&\sigma(A(x))\subset[\alpha,\beta]\quad \forall\, x\,\in \,\Omega.
\end{aligned}\end{equation}

The weak formulation of problem \eqref{model} is to find $u\in H^1_0(\Omega)$
such that
\begin{equation}\label{vaspro}
\begin{aligned}
&\int_{\Omega} A\nabla u\cdot\nabla v \dx=\int_{\Omega} fv\dx
\quad\forall \, v\in \, H^1_0(\Omega).
\end{aligned}
\end{equation}
Clearly, the Lax-Milgram lemma \cite{SuSc} implies that \eqref{vaspro} has a unique solution.

In order to deal with the multiscale problem that has singularities, we decompose the research domain $\Omega$ into two parts, $\Omega_{1}$ and $\Omega_{2}$, {where $\Omega_{1}$ consists of some subdomain(s) containing the singularities and $\Omega_{2}=\Omega \backslash \overline{\Omega}_{1}$} (see Figure~\ref{jpg1} for an illustration). {Let} $\Gamma\,=\partial\Omega_{1}\cap\partial\Omega_{2}$ {be} the interface {between} $\Omega_{1}$ and $\Omega_{2}$.
We assume that the length/area of $\Gamma$ satisfies $|\Gamma|=O(1)$, and $\Gamma$ is Lipschitz continuous.
\begin{figure}[htp]
  \begin{minipage}[t]{0.5\linewidth}
  \centerline{\includegraphics[scale=0.56]{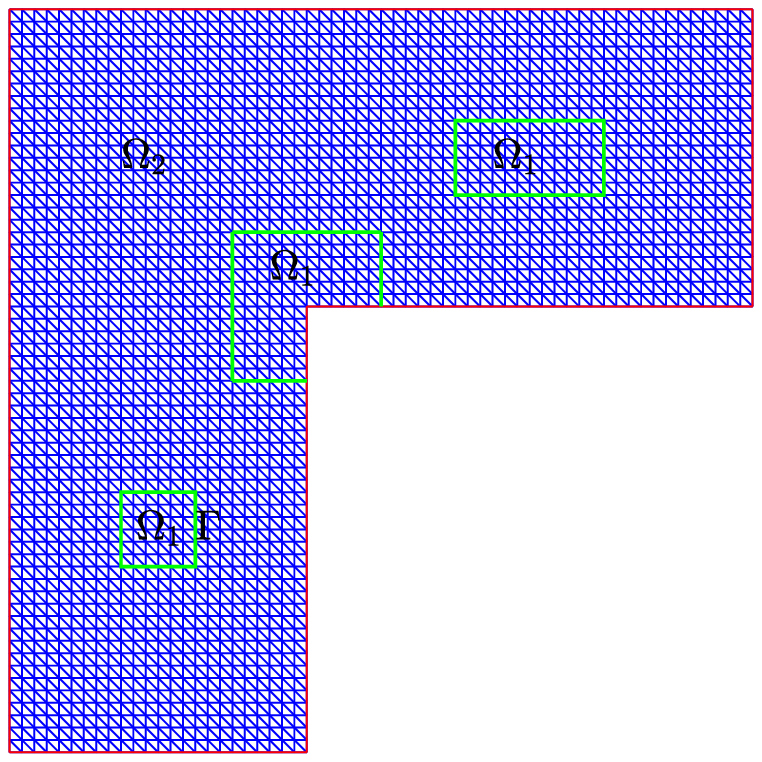}}
  \end{minipage}
  \begin{minipage}[t]{0.5\linewidth}
  \centerline{\includegraphics[scale=0.56]{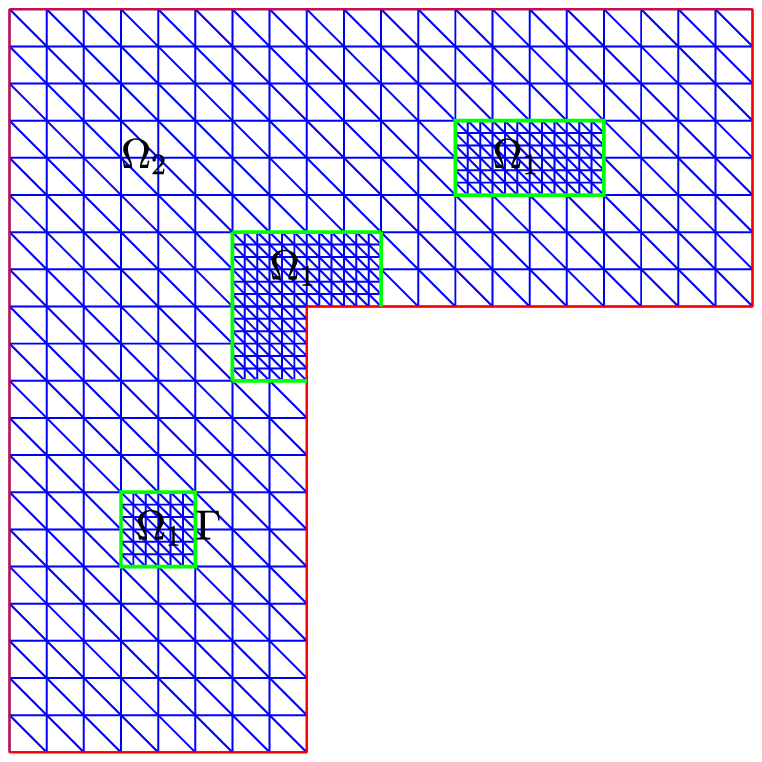}}
  \end{minipage}
\caption{{A decomposition of $\Om$ into $\Om_1$ with singularities and $\Om_2=\Om\backslash\Om_1$, where the green lines represent the interface $\Ga$. Left: A fine-scale mesh for the reference problem; Right: A mesh for the combined multiscale methods. } }
\label{jpg1}
\end{figure}

For any subdomain $\omega\subseteq\Omega$, we denote by  $(u,v)_{\om}=\int_{\om} uv$. For any segment/patch $\ga\subseteq \Ga$, denote by $\pd{u,v}_\ga$ $= \int_\ga uv$. For brevity, let $(u,v)=(u,v)_{\Om}$, $(\nabla u,\nabla v)=(\nabla u,\nabla v)_{\Om}$ and $\pd{u,v}=\pd{u,v}_\Ga$.

\subsection{Reference problem}
In this {subsection, we introduce the IPCDG method which will be used as} the fine-scale reference problem to estimate the error of our FE-LODM.

Let $\mathcal{M}_{h,\Omega_{1}}$ and $\mathcal{M}_{h,\Omega_{2}}$ be regular and quasi-uniform triangulations of $\Omega_{1}$ and $\Omega_{2}$, respectively.
Denote by $\mathcal{M}_{h,h}:=\mathcal{M}_{h,\Omega_{1}}\cup\mathcal{M}_{h,\Omega_{2}}$ the resulted triangulation of $\Om$. Note that any element in the triangulation is considered closed by convention. For any $T\in\mathcal{M}_{h,h}$, let $h_T:={\rm diam}\, T$. Denote by $h:=\max_{T\in\mathcal{M}_{h,h}}h_T$. Let $\mathbf{n}$ be the unit normal vector of $\Ga$ that points from $\Om_1$ to $\Om_2$. We
 define the jump and average of a function $v$ across $\Ga$ by $[v]:=v|_{\Om_1}-v|_{\Om_2}$ and $\{v\}:=(v|_{\Om_1}+v|_{\Om_2})/2$, respectively. Moreover, we denote by $\na_h$ the piecewise gradient on $\Om_1\cup\Om_2$, that is, $\na_hv|_{\Om_i}=\na (v|_{\Om_i}), i=1,2$.

Denote by $\Ga_i=\pa\Om\cap\pa\Om_i$ and
\eqn{H_{\Ga_i}^1(\Om_i):=\set{v\in H^1(\Om_i):\;v|_{\Ga_i}=0 \text{ in the sense of trace}}, i=1,2.}
 Let $V_{h,\Om_i}$ be the continuous linear Lagrange finite element space on $\M_{h,\Om_i}, i=1,2$,  respectively,  i.e.
 \eqn{
 V_{h,\Om_i}:=\set{v_h\in H_{\Ga_i}^1(\Om_i):\; v_h|_T\in P_1,\;\forall\, T\in\M_{h,\Om_i}},}
 where $P_1$ is the set of polynomials with total degree $\le 1$. Then the approximation space of the fine-scale reference problem is defined by
 \eq{\label{Vhh}
 V_{h,h}:=\set{v_h:\; v_h|_{\Om_i}\in V_{h,\Om_i},\;i=1,2}.}
Note that a discrete function in $V_{h,h}$ is continuous on each $\Om_i, i=1,2$, but may be discontinuous across the interface $\Ga$.
Given some positive penalty parameters $\gamma_{0}>0$, for any subset $\om\subseteq\Omega$, we define a symmetric bilinear form $a_{\om}(\cdot,\cdot)$ as follows:
\begin{align}
a_{\om}(u,v) := &(A\nabla_h u,\nabla_h v)_\om - \big(\pd{\{A\nabla_h u\cdot \mathbf{n}\},[v]}_{\Ga\cap\om}+\pd{[u],\{A\nabla_h v\cdot \mathbf{n}\}}_{\Ga\cap\om}\big)\notag\\
              &+J_{\om}(u,v),\label{aom}\\
 J_{\om}(u,v):= & \pd{\frac{\gamma_{0}}{h}[u],[v]}_{\Ga\cap\om}.\label{Jom}
\end{align}
Then the IPCDG method (cf. \cite{cao_wu}) reads as: find $u_{h,h}\in V_{h,h}$, such that
\begin{equation}\label{solution1}
a_{\Omega}(u_{h,h},v_{h,h})=(f,v_{h,h})\quad \forall\, v_{h,h} \in V_{h,h}.
\end{equation}
\begin{remark}
{\rm (1)} Noticed that $u_{h,h}$ is continuous in $\Om_2$ and subdomain(s) in $\Om_1$ and that the discontinuities across the interface $\Ga$ is treated by the interior penalty technique from the IPDG methods \cite{arnold82}, so we call this method \eqref{solution1} the IPCDG method.

{\rm (2)} It is easy to verify that the IPCDG method is consistent with the multiscale problem \eqref{model}, that is
\eqn{a_\Om(u-u_{h,h},v_{h,h})=0\quad \forall\, v_{h,h} \in V_{h,h}.}
\end{remark}

Introduce the discrete energy norm
\eqn{\normE{v}=\Big(\|A^{\frac{1}{2}}\nabla_h v\|^2_{0,\Omega}+\frac{\gamma_0}{h}\|[v]\|^2_{\Ga}+\frac{h}{\gamma_0}\|\{A\nabla_h v\cdot \mathbf{n}\}\|_\Ga^2\Big)^\frac12.
}
Clearly, the bilinear form $a_\Om$ is continuous on $V\times V$ where $V:=\{v:\; v|_{\Om_i}\in H^2(\Om_i)\cap H^1_{\Ga_i}(\Om_i), i=1,2\}$, i.e.
\eq{\label{continuous}
|a_\Om(u,v)|\ls\normE{u}\normE{v}.}
Following \cite{cao_wu,arnold82}, it may be proved that there exists a positive constant $\al_0$ such that
\eq{\label{coercive}
a_\Om(v_{h,h},v_{h,h})\gtrsim \normE{v_{h,h}} \quad v_{h,h} \in V_{h,h},\quad\text{if }\ga_0\ge\al_0,}
and hence the following C\'ea's lemma and the well-posedness hold for the IPCDG method \eqref{solution1} if the penalty parameter $\ga_0\ge\al_0$. We omitted the details.
\begin{lemma}\label{Lipcdg} Suppose $\ga_0\ge\al_0$. Then the following error estimate holds:
\eqn{
\normE{u-u_{h,h}}\ls \inf_{v_{h,h}\in V_{h,h}}\normE{u-v_{h,h}}.}
\end{lemma}
Then the error estimate of the IPCDG method may be obtained by combining the above C\'ea's lemma and the interpolation error estimates. We omitted the details and just assume that
 the IPCDG solution $u_{h,h}$ is a good approximation of the exact solution $u$. We will use $u_{h,h}$ as a reference solution to estimate the error of our FE-LODM.

For further error analysis, we introduce the following norm on the restriction of the space $V_{h,h}$ onto a subdomain $\om\subseteq\Om$:
\begin{equation}\label{norm1}
\|v\|_{h,h,\om}=\Big(\|A^{\frac{1}{2}}\nabla_h v\|^2_{0,\om}+\frac{\gamma_0}{h}\|[v]\|^2_{\Ga\cap\om}\Big)^\frac12
\end{equation}
and denote by $\|\cdot\|_{h,h}:=\|\cdot\|_{h,h,\Om}$
the norm on the whole domain $\Om$. Noting that the norm $\|\cdot\|_{h,h}$ is just the norm $\normE{\cdot}$ with the third term dropped, by using the trace and inverse inequalities \cite{SuSc}, it is easy to show that the two norms are equivalent on the fine-scale approximation space $V_{h,h}$.

\section{FE-LODM formulation}
In this section, we will present the FE-LODM which uses FEM in the domain $\Om_1$ containing singularities and LODM in $\Om_2$ where the solution is smooth but highly oscillating and the two methods are joint at the interface $\Ga$ by using the interior penalty technique. To do this, we first introduce coarse meshes on $\Om_2$ and coarse-scale finite element spaces, secondly state the multiscale decomposition of the fine-scale space $V_{h,h}$ and  the approximation space for the FE-LODM, then present the ideal combined multiscale method, and finally formulate the localized combined multiscale method, i.e., FE-LODM.

{\subsection{Coarse-scale FE spaces}

Let $\mathcal{M}_{H,\Omega_{2}}$ a shape-regular coarse triangulation of  $\Omega_{2}$ such that  the fine reference mesh $\mathcal{M}_{h,\Omega_{2}}$ is a refinement of it. Denote by} $H$  the maximum diameter of elements in $\mathcal{M}_{H,\Omega_{2}}$. Clearly, $h<H$. Let $\mathcal{M}_{\Gamma_{h}}$ and $\mathcal{M}_{\Gamma_{H}}$ be the set of interface elements in  $\mathcal{M}_{h,\Omega_{1}}$ and $\mathcal{M}_{H,\Omega_{2}}$, respectively, i.e.
\eqn{\mathcal{M}_{\Gamma_{h}}:=\set{T\in \mathcal{M}_{h,\Omega_{1}}:\; |T\cap\Ga|\neq 0}\;\text{and}\;\mathcal{M}_{\Gamma_{H}}:=\set{T\in \mathcal{M}_{H,\Omega_{2}}:\; |T\cap\Ga|\neq 0}.
}
 Denote by $\Gamma_{h}:=\set{T\cap\Ga:\;T\in \mathcal{M}_{\Gamma_{h}}}$ and $\Gamma_{H}:=\set{T\cap\Ga:\;T\in \mathcal{M}_{\Gamma_{H}}}$ the two partitions of the interface $\Gamma$ induced by $\mathcal{M}_{h,\Omega_{1}}$ and $\mathcal{M}_{H,\Omega_{2}}$, respectively.
In addition, we assume that $\mathcal{M}_{h,\Omega_{1}}$ and $\mathcal{M}_{H,\Omega_{2}}$
satisfy the matching condition that $\Gamma_{h}$ is a refinement of $\Gamma_{H}$. Introduce the coarse-scale finite element space on the coarse mesh $\mathcal{M}_{H,\Om_2}$:
\eqn{V_{H,\Omega_{2}}:=\set{v_H\in H_{\Ga_2}^1(\Om_2):\; v_H|_T\in P_1,\;\forall\, T\in\M_{H,\Om_2}}.}
Let
 \eqn{
 V_{h,H}:=\set{v_{h,H}:\; v_{h,H}|_{\Om_1}\in V_{h,\Om_1},\;v_{h,H}|_{\Om_2}\in V_{H,\Om_2}}.}
 Moreover, we denote by
 \eqn{V_{0,h}:= \{v_{h,h}\in V_{h,h} :\; v_{h,h}|_{\Omega_{1}}=0\} \;\text{and}\;V_{0,H}:= \{v_{h,H}\in V_{h,H} :\; v_{h,H}|_{\Omega_{1}}=0\}.}

\subsection{Multiscale Decomposition}
First, we need to define a quasi-interpolation operator from the fine-scale approximation space to the coarse-scale space. For this, we first introduce a weighted Cl\'{e}ment-type quasi-interpolation operator $\mathcal{C}_{H}$  defined on the region $\Omega_{2}$ (see \cite{clement,clementt}). Let $\mathcal{N}_{H}$ be the set of vertices of elements in $\mathcal{M}_{H,\Omega_{2}}$ and let  $\mathring{\mathcal{N}}_{H}:= \mathcal{N}_{H}\backslash \Ga_2$. For any node $z\in\mathcal{N}_{H}$, let  $\Phi_{z}$ $\in$ $V_{H,\Omega_{2}}$ be the nodal basis function at $z$.
The Cl\'{e}ment-type quasi-interpolation operator $\mathcal{C}_{H}$: $H_{\Gamma_{2}}^{1}(\Omega_{2})$ $\mapsto$ $V_{H,\Omega_{2}}$ is given by:
\begin{equation}\label{defclement}
 \mathcal{C}_{H}u:=\sum\limits_{{z} \in \mathring{\mathcal{N}}_{H}} u_{z}\Phi_{z} \quad
with\,\, u_{z}=\frac{(u,\Phi_{z})_{\Omega_{2}}}{(1,\Phi_{z})_{\Omega_{2}}}\qquad  \forall\, u\,\in H_{\Gamma_{2}}^{1}(\Omega_{2}).
\end{equation}
Further, let $\Pi_{h}: L^{2}(\Omega_{1}) \mapsto V_{h,\Omega_{1}}$ be the $L^{2}$-projection operator. Clearly,
\begin{equation}\label{kk77}
\Pi_{h}v_{h}=v_{h}\qquad  \forall \,v_{h}\,\in V_{h,\Omega_{1}}.
\end{equation}
Then the quasi-interpolation operator $\mathcal{C}_{h,H}$:$V_{h,h}$ $\mapsto$ $V_{h,H}$ can be defined as

\begin{equation}\label{ChH}
 \begin{cases} \mathcal{C}_{h,H}v_{h,h}|_{\Omega_{1}}:=\Pi_{h}(v_{h,h}|_{\Om_1}); \\
 \mathcal{C}_{h,H}v_{h,h}|_{\Omega_{2}}:=\mathcal{C}_{H}(v_{h,h}|_{\Om_2}),
\end{cases}\quad \text{for~any~} v_{h,h}\in V_{h,h}.\end{equation}

By the operator $\mathcal{C}_{h,H}$, we can define its kernel space $W_{h,h}:=\{v_{h,h}\in
V_{h,h}\mid\mathcal{C}_{h,H}v_{h,h}=0 \}$, and use it to construct a splitting of
the space $V_{h,h}$ into the direct sum
\begin{equation}
 V_{h,h}=V_{h,H}\oplus W_{h,h}.
\end{equation}
Notice that, for any $w_{h,h}$ $\in$ $W_{h,h}$, from \eqref{kk77} it follows that $w_{h,h}|_{\Omega_{1}}=0$. Hence, in the following, we change the notation $W_{h,h}$ into $W_{0,h}$ for emphasis.

Note that the subspace $W_{0,h}$ is a fine-scale remainder space ({high-dimensional} space), which contains the fine-scale features of $V_{0,h}$ that can not be expressed in the low-dimensional space $V_{0,H}$.
Following the idea of LODM (\cite{HenningLocalized,HenningA}), we look for a splitting $V_{h,h}=V^{ms}_{h,H}\oplus W_{0,h}$ such that
the space $V^{ms}_{h,H}$ has good $H^{1}$ approximation properties to the solution of the multiscale problem.
It is obvious that $V^{ms}_{h,H}$ is a {low-dimensional} space that it has the same dimension as $V_{h,H}$.
In order to explicitly construct such a splitting, we look for the orthogonal complement of $W_{0,h}$ in $V_{h,h}$ with respect to the scalar product $a_{\Omega}{(\cdot,\cdot)}$.

The corresponding fine-scale projection $Q_{h}$: $V_{h,h} \mapsto W_{0,h}$ is given by: for $v_{h,h}$ $\in$ $V_{h,h}$, find ${Q_{h}v_{h,h}}$ $\in$ $W_{0,h}$, such that
\begin{equation}\label{corg}
a_{\Omega}({Q_{h}v_{h,h}},w_{0,h})=a_{\Omega}(v_{h,h},w_{0,h}) \quad
\forall\, w_{0,h}\in W_{0,h}.
\end{equation}
Using the fine-scale projection, we can define {the approximation space on the whole domain $\Om$ for the ideal combined multiscale method}  by
\begin{align}\label{VmshH}
V^{ms}_{h,H}: =({I}-Q_{h})V_{h,H}.
\end{align}

\subsection{The ideal combined multiscale method}
Next, we define the ideal combined multiscale method  for the problem \eqref{solution1} as follows: for all $v^{ms}_{h,H}\in V^{ms}_{h,H}$, find $u^{ms}_{h,H}\in
V^{ms}_{h,H}$, such that
\begin{equation}\label{solutionms}
a_{\Omega}(u^{ms}_{h,H},v^{ms}_{h,H})=(f,v^{ms}_{h,H}).
\end{equation}

With above definition of the ideal method, we can see that, by taking $v^{ms}_{h,H}=u^{ms}_{h,H}$ in \eqref{solutionms} and using the coercivity \eqref{coercive} of $a_\Om$ on $V_{h,h}$, we have the following stability for the ideal solution $u^{ms}_{h,H}$: if $\ga_0\ge\al_0$,
 \eq{\label{stability}
 \|u^{ms}_{h,H}\|_{h,h}\ls \|f\|_{L^2(\Om)}.}

\begin{remark}\label{re13} It can be proved that $\mathcal{C}_{h,H}$ is an isomorphism on $V_{h,H}$ (see Lemma~\ref{inverse}). Thus we can split
 $u^{ms}_{h,H}$  into
\[
u^{ms}_{h,H}=\underbrace{(\mathcal{C}_{h,H}|_{V_{h,H}})^{-1}\mathcal{C}_{h,H}u_{h,H}^{ms}}_{\in V_{h,H}}-\underbrace{\left((\mathcal{C}_{h,H}|_{V_{h,H}})^{-1}\mathcal{C}_{h,H}u_{h,H}^{ms}-u^{ms}_{h,H}\right)}_{\in W_{0,h}}.
\]
Moreover, it is easy to check that
$$
(\mathcal{C}_{h,H}|_{V_{h,H}})^{-1}\mathcal{C}_{h,H}u_{h,H}^{ms}-u^{ms}_{h,H}=Q_h((\mathcal{C}_{h,H}|_{V_{h,H}})^{-1}\mathcal{C}_{h,H}u_{h,H}^{ms}).
$$
Therefore, we have the splitting that obeys \eqref{VmshH}
  \eq{\label{DP}u^{ms}_{h,H}=u_{h,H}-Q_h(u_{h,H}),\quad\text{where } u_{h,H}:= (\mathcal{C}_{h,H}|_{V_{h,H}})^{-1}\mathcal{C}_{h,H}u_{h,H}^{ms}.}  It is clear that $\mathcal{C}_{h,H}u^{ms}_{h,H}=\mathcal{C}_{h,H}u_{h,H}$.
\end{remark}

Before closing this subsection, we present an interesting result in the following proposition about the ideal method \eqref{solutionms} which says that there is no difference between the ideal solution and the reference solution in the subdomain $\Om_1$.
\begin{proposition}\label{prop1}
Let $u_{h,h}$ and $u_{h,H}^{ms}$ be the solutions to \eqref{solution1} and \eqref{solutionms}, respectively. Then we have
\begin{equation}
u_{h,h}-u_{h,H}^{ms}\in W_{0,h}\quad\text{and}\quad \mathcal{C}_{h,H}u_{h,h}=\mathcal{C}_{h,H}u_{h,H},
\end{equation}
where $u_{h,H}$ is defined in Remark~\ref{re13}. Especially, $u_{h,h}|_{\Om_1}=u_{h,H}^{ms}|_{\Om_1}$.
\end{proposition}
\begin{proof}
Form \eqref{solution1} and (\ref{solutionms}), for all $v_{h,H}^{ms}\in V_{h,H}^{ms}$, it follows that
\begin{align*}
&a_{\Omega}(u_{h,h},v_{h,H}^{ms})=(f,v_{h,H}^{ms}),\\
&a_{\Omega}(u_{h,H}^{ms},v_{h,H}^{ms})=(f,v_{h,H}^{ms}).
\end{align*}
Thus
\begin{equation}\label{glaerkinorth}
a_{\Omega}(u_{h,h}-u_{h,H}^{ms},v_{h,H}^{ms})=0,
\end{equation}
which means $u_{h,h}-u_{h,H}^{ms}\in W_{0,h}$. Hence
\[
\mathcal{C}_{h,H}u_{h,h}=\mathcal{C}_{h,H}u_{h,H}^{ms}=\mathcal{C}_{h,H}u_{h,H},
\]
which yields the results immediately.
\end{proof}

\subsection{Formulation of FE-LODM}\label{sec34}

Note that the fine-scale projection $Q_{h}$ in \eqref{corg} is defined globally onto $W_{0,h}$, and consequently, in order to calculating the basis functions of the discrete space $V^{ms}_{h,H}$, one has to solve many large equations with ${\rm dim}(W_{0,h})$ unknowns. Therefore, for practical application, we have to localize the definition of $Q_h$ to obtain an approximation of the ideal combined multiscale method, i.e., FE-LODM. To do so, we first decompose $Q_h$ by restrict \eqref{corg} to each elements in $\mathcal{M}_{h,H}$ and then localize the restrictions.

For each $T\in \mathcal{M}_{H,\Omega_{2}}$ we associate it with a point set $\tilde T\supseteq T$ defined as follows. If $T\in \mathcal{M}_{H,\Omega_{2}}\backslash\mathcal{M}_{\Gamma_H}$, we just let $\tilde T =T$. While, for $T\in \mathcal{M}_{\Gamma_H}$,  we let $$\tilde T:=T\cup\set{t\in \mathcal{M}_{\Gamma_h}:\; (t\cap\Ga)\subset(T\cap\Ga)}$$ be the union of $T$ and those interface elements in $\mathcal{M}_{\Gamma_h}$ whose intersections with the interface $\Ga$ are contained in $\pa T$ (see Figure~\ref{elepatch} for an illustration).
\begin{figure}[htp]
\centerline{\includegraphics[width=0.56\textwidth]{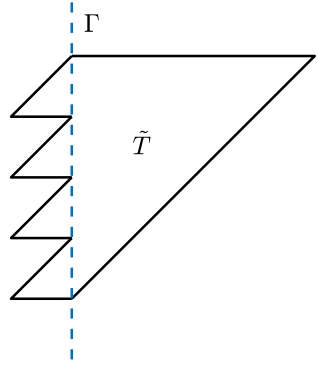}}
\caption{An illustration of the interface  combined element $\tilde{T}$ for $T\in \mathcal{M}_{\Ga_H}$. }
\label{elepatch}
\end{figure}

Then we define the restrictions of $Q_h$ for any $T\in \mathcal{M}_{H,\Omega_{2}}$ as:
   ${Q_h^T}v_{h,h}$ $\in$ $W_{0,h}$ such that
\begin{equation}\label{ahti}
a_{\Omega}(Q_h^Tv_{h,h},w_{0,h})=a_{\tilde T}(v_{h,h},w_{0,h}) \quad
\forall\, w_{0,h}\in W_{0,h}.
\end{equation}
\begin{remark}\label{re11}
For $v_{h,h}\in V_{h,h}, w_{0,h}\in W_{0,h}$, $a_\Om(v_{h,h}, w_{0,h})$ has nonzero terms on the interface $\Ga_h$. We integrate them into the definition of ${Q_h^T}v_{h,h}$ for the elements $T\in\mathcal{M}_{\Gamma_H}$ by restricting the bilinear form $a$ onto $\tilde T$.
\end{remark}

Noting that any function in $W_{0,h}$ vanishes in $\Om_1$, from \eqref{aom} we have
\eqn{a_{\Om_1\backslash (\cup_{T\in\M_{\Ga_h}}T)}(v_{h,h},w_{0,h})=0\quad\forall\, v_{h,h}\in V_{h,h}, w_{0,h}\in W_{0,h}.}
Therefore it follows from \eqref{corg} that
\eqn{a_{\Omega}(Q_{h}v_{h,h},w_{0,h})=\sum_{T\in\mathcal{M}_{H,\Om_2}}a_{\tilde T}(v_{h,h},w_{0,h})\quad
\forall\, w_{0,h}\in W_{0,h},}
and hence we have the following decomposition:
\eq{\label{QhD} Q_h=\sum_{T\in\mathcal{M}_{H,\Om_2}}Q_h^T.}

Although the definitions of $Q_h^Tv_{h,H}$ are independent of each other and can be computed in parallel, they are still global and have to be solved on the whole fine-scale space $W_{0,h}$.

Next we introduce local versions of the correction operators $Q_h^T$  and give error estimates between them. To this end, we need the definitions of element patches (c.f. \cite{HenningLocalized}).
\begin{definition}[element patches]
 {Given $T\in \mathcal{M}_{H,\Omega_{2}}$, the patches $T_L$ are defined recursively as follows:}
\begin{align*}
&T_0:=T,\\
&T_L:=\set{T'\in \mathcal{M}_{H,\Omega_{2}}|\,\,T'\cap {T_{L-1}}\neq \emptyset}\quad L=1,2,\cdots.
\end{align*}
\end{definition}
The restriction of the fine-scale correction space $W_{0,h}$ to the element patch $T_L$ is defined by
\eqn{W_{0,h}(T_L) :=\{v_{h,h}\in
W_{0,h}:\,\,v_{h,h}=0\,\,\text{ in }\,\,\Omega\backslash T_L\}.}
The localized approximation of the correction operator $Q_h^T$ is defined as follows.
\begin{definition}
For ${T}\in \mathcal{M}_{H,\Omega_{2}}$ and {the patch $T_L$,} the local correction operator $Q_h^{T,L}$: $V_{h,h}$ $\mapsto$ $W_{0,h}(T_L)$ is defined as follows: given $v_{h,h}$ $\in$ $V_{h,h}$, find $Q_h^{T,L}v_{h,h}$ $\in$ $W_{0,h}(T_L)$  such that
\begin{equation}\label{corl}
a_{\Om}(Q_h^{T,L}v_{h,h},w_{0,h})=a_{\tilde T}(v_{h,h},w_{0,h}) \quad
\forall\, w_{0,h}\in W_{0,h}(T_L).
\end{equation}
\end{definition}
According to {the decomposition \eqref{QhD} and} the above definition, the global corrector of level $L$ is given by
\begin{equation}\label{QhL}
{Q^{L}_h:=\sum\limits_{T\in \mathcal{M}_{H,\Omega_{2}}} Q_h^{T,L}.}
\end{equation}
Further, we define the {localized} multiscale approximation space as follows:
\begin{align}\label{VhmsL}
V^{ms,L}_{h,H}: =\big({I}-Q_{h}^{L}\big)V_{h,H} {=\set{v_{h,H}-Q^{L}_{h}v_{h,H}:\;v_{h,H}\in V_{h,H}}}.
\end{align}
Then, the FE-LODM {reads as: find}
$u^{ms,L}_{h,H}\in V^{ms,L}_{h,H}$, such that
\begin{equation}\label{solution2}
a_{\Omega}\big(u^{ms,L}_{h,H},v^{ms,L}_{h,H}\big)=\big(f,v^{ms,L}_{h,H}\big)\quad
\forall\, v^{ms,L}_{h,H}\in V^{ms,L}_{h,H}.
\end{equation}
\begin{remark}
{\rm (1)} {It is clear that $V^{ms,L}_{h,H}\subset V_{h,h}$ as a consequence of }the assumption that
 $\mathcal{M}_{h,\Omega_{2}}$ is a refinement of $\mathcal{M}_{H,\Omega_{2}}$. Therefore the FE-LODM inherits the well-posedness from the reference problem \eqref{solution1}.

{\rm (2)} {Unlike $V^{ms}_{h,H}$ whose multiscale basis functions supported globally on $\Om_2$, the multiscale basis functions of $V^{ms,L}_{h,H}$ locally support on small patches of size $O(LH)$,  and hence the computational cost for assembling the global system of the FE-LODM \eqref{solution2} is usually much less than that for the ideal combined method \eqref{solutionms}.
}
\end{remark}

\section{Error estimates for the FE-LODM}
In this section, we derive the $H^{1}$ and $L^{2}$ error estimates for the proposed FE-LODM.

We first recall two {local} trace inequalities which will be used in this paper frequently. Here we omitted the proof since it is a direct consequence of the standard trace inequality (cf. \cite[Theorem 1.6.6, p.39]{SuSc}) and the scaling argument (cf. \cite{PG}).
\begin{lemma}\label{lem2}
Let $T$ be an element {in the triangulation $\mathcal{M}_{h,\Omega_{1}}$, $\mathcal{M}_{h,\Omega_{2}}$, or $\mathcal{M}_{H,\Omega_{2}}$.} Then, we have
\begin{align*}
\|v\|_{0,\partial T}&\lesssim {\rm diam}(T)^{-\frac{1}{2}}\|v\|_{0,T}+\|v\|^{\frac{1}{2}}_{0,T}\|\nabla v\|^{\frac{1}{2}}_{0,T}\quad \forall\, v\,\in \, H^{1}(T),\\
\|v\|_{0,\partial T}&\lesssim  {\rm diam}(T)^{-\frac{1}{2}}\|v\|_{0,T} \quad \forall\, v\,\in\, P_{1}(T),
\end{align*}
{where the invisible constants depend only} on the regularity of the element $T$.
\end{lemma}

We shall also make use of the following {norm on the restriction of the space $V_{h,H}$ onto a subdomain $\om\subseteq\Om$:
\begin{equation}\label{norm2}
\|v\|_{h,H,\om}=\Big(\|A^{\frac{1}{2}}\nabla_h v\|^2_{0,\om}+\frac{\gamma_0}{H}\|[v]\|^2_{\Ga\cap\om}\Big)^\frac12,
\end{equation}
and denote by $\|\cdot\|_{h,H}:=\|\cdot\|_{h,H,\Om}$. Note that the norm $\|\cdot\|_{h,H,\om}$ is almost the same as the previous one $\|\cdot\|_{h,h,\om}$ in \eqref{norm1} except replacing $\frac{\ga_0}{h}$ there by $\frac{\ga_0}{H}$.}

\subsection{Properties of the operator $\mathcal{C}_{h,H}$}
In this subsection, we state two lemmas on the quasi-interpolation operator $\mathcal{C}_{h,H}$ given in \eqref{ChH}.

First we recall  some  stability and error estimates for the operator $\mathcal{C}_{H}$, whose proof can be found in \cite{clement,clementt}.
\begin{lemma}\label{lemstab}
For any $T\in\mathcal{M}_{H,\Omega_{2}}$ and ${v}_{h}\in V_{h,\Omega_{2}}$, {there hold following estimates}
\begin{align}
\|\nabla_h \mathcal{C}_{H}v_{h}\|_{0,T}&\lesssim \|\nabla_h v_{h}\|_{0,\hat{T}},\label{k1}\\
\|v_{h}-\mathcal{C}_{H}v_{h}\|_{0,T}&+H\|\nabla_h(v_{h}-\mathcal{C}_{H}v_{h})\|_{0,T}\lesssim H\|\nabla_h v_{h}\|_{0,\hat{T}},\label{k2}
\end{align}
where $\hat{T}=\cup\{T^{'}\in\mathcal{M}_{H,\Omega_{2}}:\; T^{'}\cap{T}\neq\emptyset\}.$
\end{lemma}

Using Lemma~\ref{lemstab}, we {may prove the following stability result for} $\mathcal{C}_{h,H}$.
\begin{lemma}\label{lem}
For any ${v_{h,h}}$ $\in$ $V_{h,h}$, it holds that
\begin{equation}\label{c3}
\|\mathcal{C}_{h,H}v_{h,h}\|_{h,H,\om} \lesssim \|v_{h,h}\|_{h,H,\hat\om},
\end{equation}
where $\hat\om:=\bigcup\big(\{\hat T:\; T\cap\om\neq\emptyset, T\in\mathcal{M}_{H,\Omega_{2}}\}\cup\{T:\; T\cap\om\neq\emptyset, T\in\mathcal{M}_{h,\Omega_{1}}\}\big)$.
\end{lemma}
\begin{proof}
From the definitions \eqref{norm1} and \eqref{norm2} of the norms, \eqref{kk77}, \eqref{ChH}, and  \eqref{k1}, it is easy to see that
\begin{align}\label{k4lem}
\|\mathcal{C}_{h,H}v_{h,h}\|^{2}_{h,H,\om}
\lesssim {}& \sum\limits_{T \in \mathcal{M}_{H,\Omega_{2}}\cap\om}\|A^{\frac{1}{2}}\nabla_h v_{h,h}\|^{2}_{0,\hat{T}}\notag\\
&+\|A^{\frac{1}{2}}\nabla_h v_{h,h}\|^{2}_{0,\Omega_{1}\cap\om}+
\frac{\gamma_{0}}{H}\|[\mathcal{C}_{h,H}v_{h,h}]\|^{2}_{{\Ga\cap\om}}\nn  \\
\lesssim {}& {\|A^{\frac{1}{2}}\nabla_h v_{h,h}\|^{2}_{0,\Omega\cap\hat\om}+\frac{\gamma_{0}}{H}\|[\mathcal{C}_{h,H}v_{h,h}]\|^{2}_{\Ga\cap\om}}.
\end{align}
For the second term on the right hand side, using the triangle inequality, we have
\begin{align}
\frac{\gamma_{0}}{H}\|[\mathcal{C}_{h,H}v_{h,h}]\|^{2}_{\Ga\cap\om}\lesssim{}& {\frac{\gamma_{0}}{H}\|[\mathcal{C}_{h,H}v_{h,h}-v_{h,h}]\|^{2}_{\Ga\cap\om}+
\frac{\gamma_{0}}{H}\|[v_{h,h}]\|^{2}_{\Ga\cap\om}} \nn \\
\lesssim {}&\sum\limits_{\substack{E\in \Gamma_{H}\\E\cap\om\neq\emptyset}}\frac{\gamma_{0}}{H}\|\mathcal{C}_{H}{v_{h,\Om_2}}-v_{h,\Om_2}\|^{2}_{E}
+\frac{\gamma_{0}}{H}\|[v_{h,h}]\|^{2}_{\Ga\cap\om}\label{eq45}\\
:={}&\mathrm{\uppercase \expandafter {\romannumeral 1}}+\frac{\gamma_0}H\|[v_{h,h}]\|^{2}_{\Ga\cap\om}, \nn
\end{align}
where $v_{h,\Om_2}$: =$v_{h,h}|_{\Omega_{2}}$, and we have used the fact that $\mathcal{C}_{h,H}v_{h,h}|_{\Omega_{1}}=v_{h,h}|_{\Omega_{1}}$.

Further, from Lemma \ref{lem2} and \eqref{k2}, it follows that
\begin{align*}
\mathrm{\uppercase \expandafter {\romannumeral 1}} &\lesssim\sum\limits_{\substack{T\in \mathcal{M}_{\Gamma_{H}}\\ T\cap\om\neq\emptyset}}\frac{\gamma_{0}}{H}\left(H^{-1}
\|\mathcal{C}_{H}v_{h,\Om_2}-v_{h,\Om_2}\|^{2}_{0,T}
+H\|\nabla_h (\mathcal{C}_{H}v_{h,\Om_2}-v_{h,\Om_2})\|^{2}_{0,T}\right)\\
&\lesssim \sum_{\substack{T\in \mathcal{M}_{\Gamma_{H}}\\ T\cap\om\neq\emptyset}}\|\nabla_h v_{h,\Om_2}\|^{2}_{0,\hat{T}},
\end{align*}
which together with \eqref{k4lem} and \eqref{eq45} yields the result immediately.
\end{proof}

The following lemma gives a stability estimate of $\mathcal{C}_{h,H}|_{V_{h,H}}$, whose proof is arranged in \ref{a:inverse} for the convenience of the reader.
\begin{lemma}\label{inverse}$\mathcal{C}_{h,H}$ is an isomorphism on $V_{h,H}$ and satisfies the following estimate
\eqn{\big\|\big(\mathcal{C}_{h,H}|_{V_{h,H}}\big)^{-1}v_{h,H}\big\|_{h,H}\ls \norm{v_{h,H}}_{h,H}\quad \forall\, v_{h,H}\in V_{h,H}.}
\end{lemma}

The following lemma is crucial for  the error analysis, which can be proved by following the proof of \cite[Lemma 2.1]{MLocalization} or \cite[Lemma 1]{MR3587383}. We omit the details.
\begin{lemma}\label{cc}
For each ${v}_{0,H}$ $\in$ $V_{0,H}$, there exists a $v_{0,h}$ $\in$ $V_{0,h}$, such that
$\mathcal{C}_{h,H}v_{0,h}=v_{0,H}$, $\|v_{0,h}\|_{h,h}$ $\lesssim$ $\|v_{0,H}\|_{h,H}$ and ${{\rm supp}\,v_{0,h}}\subseteq {\rm supp}\,(\mathcal{C}_{h,H}v_{0,H})$.
\end{lemma}
We emphasize that the above result holds for any function in $V_{0,H}$, not $V_{h,H}$, which is sufficient for the later analysis. In fact, we have tried to use $V_{h,H}$ instead of $V_{0,H}$, but the error estimate has become worse, multiplying by an additional factor  $H/h$.

\subsection{Error estimate for the ideal combined multiscale method}
The following theorem gives an error bound for the ideal multiscale method \eqref{solutionms}, where the correctors for the basis functions have to be solved globally (see \eqref{corg} and \eqref{VmshH}). The proposed  ideal combined multiscale method preserves the common linear order convergence {$O(H)$} for the $H^{1}$-error without suffering from preasymptotic effects due to the highly varying diffusion coefficient.
\begin{theorem}\label{thm:ideal1}
If $u_{h,h}$ and $u^{ms}_{h,H}$ are the solutions of the reference problem \eqref{solution1} and the approximation problem \eqref{solutionms} respectively, then it holds that
\begin{equation}\label{h15}
\|u_{h,h}-u^{ms}_{h,H}\|_{h,h}\lesssim H\|f\|_{L^{2}(\Omega)}.
\end{equation}
\end{theorem}
\begin{proof}
Let $e_{h}=u_{h,h}-u^{ms}_{h,H}$. From Proposition~\ref{prop1} we have $e_{h}$ $\in$ $W_{0,h}$. Hence ${C}_{h,H}e_{h}=0$. Thus, using the Cauchy-Schwarz inequality and the coercivity \eqref{coercive} of $a_\Om$ on $V_{h,h}$, we have
\begin{align*}
\|e_{h}\|^{2}_{h,h}&\lesssim a_{\Omega}(e_{h},e_{h})=(f,e_{h})
=(f,e_{h}-\mathcal{C}_{h,H}e_{h})\\
&\leq \|f\|_{L^{2}(\Omega)}\|e_{h}-\mathcal{C}_{h,H}e_{h}\|_{L^{2}(\Omega)}\\
&=\|f\|_{L^{2}(\Omega)}\sum\limits_{T \in \mathcal{M}_{H,\Omega_{2}}}\|e_{h}-\mathcal{C}_{H}e_{h}\|_{0,T}.
\end{align*}
Further, from \eqref{k2}, it follows that
\begin{align*}
\sum\limits_{T \in \mathcal{M}_{H,\Omega_{2}}}\|e_{h}-\mathcal{C}_{H}e_{h}\|_{0,T}
&\lesssim \sum\limits_{T \in \mathcal{M}_{H,\Omega_{2}}} H\|\nabla e_{h}\|_{0,\hat{T}}\lesssim H\|e_{h}\|_{h,h},
\end{align*}
which combines the above estimate yields the result immediately.
\end{proof}

\subsection{Error estimates for the localized method}
In this subsection we first estimate the errors between the correctors $Q_h^T$ and $Q_h^{T,L}$ due to the truncations to local patches. Then we provide $H^1$-~and~$L^2$- error bounds for the FE-LODM.

We will frequently make use of the following cut-off functions on element patches:
for each $T$ $\in$ $M_{H,\Omega_{2}}$ and $l_1<l_2$ $\in$ $\mathbb{N}$, the cut-off functions $\eta^{l_1,l_2}_{T}$ $\in$ $V_{h,H}$ satisfy
\begin{align}\label{bb}
&\eta^{l_1,l_2}_{T}|_{T_{l_1}}=1,\\
&\eta^{l_1,l_2}_{T}|_{\Omega \backslash T_{l_2}}=0,\label{bb2}\\
\|\nabla_h &\eta^{l_1,l_2}_{T}\|_{L^{\infty}(\Omega)}\lesssim \frac{1}{(l_2-l_1)H_{T}}.\label{66}
\end{align}

Let $I_{h,h}$ ( $I_{h,h}|_{\Omega_{i}}: H_{\Gamma_{i}}^{1}(\Omega_{i})\cap C(\bar{\Omega}_{i})\mapsto V_{h,\Omega_{i}}, i=1,2$ ) be the linear Lagrange interpolation operator with respect to $\mathcal{M}_{h,h}$. The following lemma provides a stability estimate of the operator $I_{h,h}$.
\begin{lemma}\label{lem8}
For $T\in\mathcal{M}_{H,\Omega_{2}}$, assume that $\eta_T^{s,n}, n>s>0\in\mathbb{N}$ is the cut-off function which satisfies \eqref{bb}--\eqref{66}. Then for $w\in W_{0,h}$, the following estimates hold
\begin{align}
\|I_{h,h}(\eta_T^{s,n}w)\|_{h,h}&\lesssim \|w\|_{h,h,T_{n+1}}\label{lem81},\\
\|\eta_T^{s,n}w-I_{h,h}(\eta_T^{s,n}w)\|_{h,h}&\lesssim \|w\|_{h,h,T_{n+1}\backslash T_{s-1}},\label{lem82}\\
\|I_{h,h}(\eta_T^{s,n}w)\|_{h,h,T_{n}\backslash T_{s}}&\lesssim \|w\|_{h,h,T_{n+1}\backslash T_{s-1}}\label{lem83},\\
\|I_{h,h}(1-\eta_T^{s,n})w\|_{h,h}&\lesssim \|w\|_{h,h,\Om\backslash T_{s-1}}\label{lem84}.
\end{align}
\end{lemma}
The proof of this lemma is similar to that of  \cite[Lemma A.2]{HenningLocalized}, except that we have to deal with the elements near the interface between coarse and fine meshes. For convenience of the reader, we arrange it in \ref{a:lem8}.

The following key lemma says that the errors of the localized correction problems decay exponentially with respect to the number of truncation layers $L$.
\begin{lemma}\label{hh12}
Let $u_{h,h}$ be the reference solution to \eqref{solution1} and $u_{h,H}^{ms}\in V_{h,H}^{ms}$ be the ideal solution to \eqref{solutionms}, respectively. Denote by $u_{h,H}=\CC u_{h,H}^{ms}\in V_{h,H}.$ Further, for $T\in \mathcal{M}_{H,\Omega_{2}}$ and its element patch $T_L$,
let $q^T_h=Q_h^T(u_{h,H})$ and $q_h^{T,L}=Q_h^{T,L}(u_{h,H})$ be the global and local multiscale-corrected {solution} obtained in \eqref{ahti} and \eqref{corl}, respectively. Then there exists a constant $0<\theta<1$ independent of $L, h, H$, and $T$, such that
\begin{equation*}
\|q^T_h-q_h^{T,L}\|_{h,h}\lesssim\theta^{L}\|u_{h,H}\|_{h,h,\tilde{T}},
\end{equation*}
where
$\tilde{T}$ is defined in Section \ref{sec34}.
\end{lemma}
The proof of this lemma is similar to that given in \cite{MLocalization} and \cite{HenningLocalized}, but with some special details related to the interface elements need to be accounted for. To make the error analysis clearer, we  arrange the proof of this lemma in \ref{a:hh12}.

The following theorem gives the $H^1$-error estimate for the FE-LODM. Using this theorem, we can quantify how many truncation layers in the localization patches can ensure the linear convergence of $O(H)$.
\begin{theorem}\label{theo} Suppose $\ga_0\ge\al_0$.
Let $u_{h,h}$ and $u^{ms,L}_{h,H}$ be the reference solution to \eqref{solution1} and the solution to the FE-LODM \eqref{solution2}, respectively. Then we have
\begin{equation}\label{err1:FELODM}
\|u_{h,h}-u^{ms,L}_{h,H}\|_{h,h}\lesssim
H\|f\|_{L^{2}(\Omega)}+\Big(\frac{H}{h}\Big)^\frac12L^{\frac{d}{2}}\theta^{L}\|f\|_{L^{2}(\Omega)},
\end{equation}
where $0<\theta<1$ is given in Lemma~\ref{hh12}. Moreover,  there exists a positive constant $L_0$ such that when $L\ge L_0|\log (Hh)^\frac12|$, we have the following estimate, which is of the same order as the ideal multiscale method,
\eq{\label{err2:FELODM}\|u_{h,h}-u^{ms,L}_{h,H}\|_{h,h}\lesssim H\|f\|_{L^{2}(\Omega)}.}
\end{theorem}
\begin{proof}
Let $u_{h,H}^{ms}\in V_{h,H}^{ms}$ be the solution to the ideal method \eqref{solutionms} using the global basis. From \eqref{DP} and \eqref{QhD}, the ideal solution can be rewritten as follows:
\eqn{
u_{h,H}^{ms}= u_{h,H}-
\sum_{T\in \mathcal{M}_{H,\Om_2}}Q_h^T(u_{h,H}),
}
where $u_{h,H}=\CC u_{h,H}^{ms} \in V_{h,H}$.
Denote by
\eqn{\widetilde{u}^{ms,L}_{h,H}&:=u_{h,H}-\sum_{T\in \mathcal{M}_{H,\Om_2}}Q_h^{T,L}(u_{h,H}),\\
z&:=\sum_{T\in \mathcal{M}_{H,\Omega_{2}}}\big(Q_h^T(u_{h,H})-Q_h^{T,L}(u_{h,H})\big).}
 It is easy to see that
\begin{equation}\label{eA}
\|u^{ms}_{h,H}-\widetilde{u}^{ms,L}_{h,H}\|_{h,h}=\|z\|_{h,h}.
\end{equation}
According to Lemma \ref{cc} with $v_{0,H}=\mathcal{C}_{h,H}I_{h,h}\big(z-\eta_T^{L+2,L+3}z\big)$, there exists a function $b\in V_{0,h}$ such that
\begin{align}
\mathcal{C}_{h,H}(b)=\mathcal{C}_{h,H}I_{h,h}\big(z-&\eta_T^{L+2,L+3}z\big),\label{eq11}\\
\|b\|_{h,h} \lesssim \big\|\mathcal{C}_{h,H}I_{h,h}\big(z-\eta_T^{L+2,L+3}z\big)\big\|_{h,H}
&=\big\|\mathcal{C}_{h,H}I_{h,h}\big(\eta_T^{L+2,L+3}z\big)\big\|_{h,H}, \label{eq12}
\end{align}
where we have used $\mathcal{C}_{h,H}I_{h,h}z=0$ to derive the last equality, which is a consequence of the fact that $z\in W_{0,h}$. 
From \eqref{bb2}, we have
$$\mathcal{C}_{h,H}I_{h,h}\big(z-\eta_T^{L+2,L+3}z\big)=-\mathcal{C}_{h,H}I_{h,h}\big(\eta_T^{L+2,L+3}z\big)=0\quad\text{in}\quad \Om\backslash T_{L+4},$$
which together with Lemma~\ref{cc}, implies that 
$${\mathrm{supp}}(b)\subseteq {\mathrm{supp}}\big({\mathcal{C}_{h,H}^2}I_{h,h}\big(z-\eta_T^{L+2,L+3}z\big)\big)\subseteq T_{L+5}\backslash T_L.$$
Therefore, from (\ref{eq11}), we have $I_{h,h}\big({z-}\eta_T^{L+2,L+3}z\big)-b\in W_{0,h}$ . Hence, from (\ref{ahti}) it follows that
\begin{equation}\label{eqor}
a_{\Omega}\big(Q_h^Tu_{h,H},I_{h,h}(z-\eta_T^{L-3,L-2}z)-b\big)=
a_{\tilde T }\big(u_{h,H},I_{h,h}(z-\eta_T^{L+2,L+3}z)-b\big)=0.
\end{equation}
Further, from $\mathrm{supp} (I_{h,h}\big({z-}\eta_T^{L+2,L+3}z\big)-b) \subset \Om_2\backslash T_L$, it follows that
\begin{equation*}
a_{\Omega}\big(Q_h^{T,L}u_{h,H},I_{h,h}(z-\eta_T^{L+2,L+3}z)-b\big)=0,
\end{equation*}
which combines with \eqref{eqor} yields
\begin{equation}\label{eq13}
a_{\Omega}\big((Q_h^T-Q_h^{T,L})u_{h,H}, z-I_{h,h}(\eta_T^{L+2,L+3}z)-b\big)=0.
\end{equation}
Therefore, from \eqref{coercive}, it follows that
\begin{align*}
\|z\|^{2}_{h,h}&\lesssim a_{\Omega}(z,z)\\
&=\sum\limits_{T \in\mathcal{M}_{H,\Omega_{2}}}a_{\Omega}\big({Q_h^Tu_{h,H}-Q_h^{T,L}u_{h,H},z}\big)\\
&=\sum\limits_{T\in \mathcal{M}_{H,\Omega_{2}}}a_{\Omega}\big((Q_h^T-Q_h^{T,L})u_{h,H},I_{h,h}(\eta_T^{L+2,L+3}z)+b\big).
\end{align*}
Further, using {\eqref{eq12} and Lemmas~\ref{lem} and} \ref{lem8},  we obtain
\begin{align*}
\|z\|^{2}_{h,h}
&\lesssim \sum\limits_{T \in\mathcal{M}_{H,\Omega_{2}}}\|(Q_h^T-Q_h^{T,L})u_{h,H}\|_{h,h}\big(\|I_{h,h}(\eta_T^{L+2,L+3}z)\|_{h,h}+\|b\|_{h,h}\big)\\
&{\lesssim \sum\limits_{T \in\mathcal{M}_{H,\Omega_{2}}}\|(Q_h^T-Q_h^{T,L})u_{h,H}\|_{h,h}\|I_{h,h}(\eta_T^{L+2,L+3}z)\|_{h,h}}\\
&\lesssim \sum\limits_{T \in\mathcal{M}_{H,\Omega_{2}}}\|(Q_h^T-Q_h^{T,L})u_{h,H}\|_{h,h}\|z\|_{h,h,T_{L+4}}.
\end{align*}
Thus, by use of the Cauchy-Schwarz inequality, we have
\begin{align*}
\|z\|^{2}_{h,h}&\lesssim \bigg(\sum\limits_{T \in\mathcal{M}_{H,\Omega_{2}}}\|(Q_h^T-Q_h^{T,L})u_{h,H}\|_{h,h}^{2}\bigg)^{\frac{1}{2}}
\bigg(\sum\limits_{T \in\mathcal{M}_{H,\Omega_{2}}}\|z\|_{h,h,T_{L+4}}^{2}\bigg)^{\frac{1}{2}}\\
&\lesssim \bigg(\sum\limits_{T \in\mathcal{M}_{H,\Omega_{2}}}\|(Q_h^T-Q_h^{T,L})u_{h,H}\|_{h,h}^{2}\bigg)^{\frac{1}{2}}\left(L^{\frac{d}{2}}\|z\|_{h,h}\right),
\end{align*}
which yields
\[
\|z\|_{h,h}^{2}\lesssim L^{d}\sum\limits_{T \in\mathcal{M}_{H,\Omega_{2}}}\|(Q_h^T-Q_h^{T,L})u_{h,H}\|_{h,h}^{2}.
\]
 According to Lemma \ref{hh12}, we have
\begin{align*}
\|z\|_{h,h}^{2}&\lesssim L^{d}\theta^{2L}\sum\limits_{T \in\mathcal{M}_{H,\Omega_{2}}}\|u_{h,H}\|_{h,h,\tilde{T}}^{2} \lesssim L^{d}\theta^{2L}\|u_{h,H}\|_{h,h}^{2}.
\end{align*}
Moreover, from Lemmas~\ref{inverse} and \ref{lem} and the stability estimate \eqref{stability} of $u_{h,H}^{ms}$, we have
\begin{align*}
\|u_{h,H}\|_{h,h}^{2}&\lesssim \frac{H}{h}\|u_{h,H}\|_{h,H}^{2}=\frac{H}{h}\|(\C_{h,H}|_{V_{h,H}})^{-1}\C_{h,H}u_{h,H}^{ms}\|_{h,H}^{2}\\
&\ls \frac{H}{h}\|\C_{h,H}u_{h,H}^{ms}\|_{h,H}^{2}\ls\frac{H}{h}\|u_{h,H}^{ms}\|_{h,H}\ls \frac{H}{h}\|f\|_{L^{2}(\Omega)}.
\end{align*}
Thus,
\begin{equation}\label{eq211}
\|z\|_{h,h}\lesssim \Big(\frac{H}{h}\Big)^\frac12 L^{\frac{d}2}\theta^{L}\|f\|_{L^{2}(\Omega)}.
\end{equation}
Noting that $V_{h,H}^{ms,L}\subseteq V_{h,h}$, from the continuity and coercivity of $a_\Om$ \eqref{continuous}--\eqref{coercive}, we have the following estimate of C\'ea lemma type:
\begin{align*}
\|u_{h,h}-u^{ms,L}_{h,H}\|_{h,h}\lesssim \inf\limits_{v^{ms,L}_{h,H}\in V_{h,H}^{ms,L}}\|u_{h,h}-v^{ms,L}_{h,H}\|_{h,h},
\end{align*}
which {implies that}
\eqn{\|u_{h,h}-u^{ms,L}_{h,H}\|_{h,h}\lesssim  \|u_{h,h}-\widetilde{u}^{ms,L}_{h,H}\|_{h,h}.}
Thus, we have
\begin{align}
\begin{split}\label{eA1}
\|u_{h,h}-u^{ms,L}_{h,H}\|_{h,h}&\leq
\|u_{h,h}-u^{ms}_{h,H}\|_{h,h}+\|u^{ms}_{h,H}-\widetilde{u}^{ms,L}_{h,H}\|_{h,h}\\
&=\|u_{h,h}-u^{ms}_{h,H}\|_{h,h}+\|z\|_{h,h},
\end{split}
\end{align}
which combines \eqref{h15} and \eqref{eq211} yields the estimate \eqref{err1:FELODM} immediately.

It remains to prove \eqref{err2:FELODM}. Let $\theta_1=\frac{1+\theta}{2}$. Noting that $L^{\frac{d}{2}}\theta^L\ls \theta_1^L$, we have
\eqn{\Big(\frac{H}{h}\Big)^\frac12L^{\frac{d}{2}}\theta^L\ls \Big(\frac{H}{h}\Big)^\frac12\theta_1^L\ls H, \quad\text{if } L\ge \frac{|\log (Hh)^\frac12|}{|\log\theta_1|},}
which implies that \eqref{err2:FELODM} holds. This completes the proof of the theorem.
\end{proof}
\begin{remark}\label{logH}
{\rm (1)} If $h=H^m$ for some constant $m>1$, then $|\log (Hh)^\frac12|\eqsim|\log H|$, and hence the condition $L\ge L_0|\log (Hh)^\frac12|$ becomes $L\ge L_0'|\log H|$ for some sufficiently large constant $L_0'$. Note that this is a standard condition for LOD type methods
\cite{MLocalization,HenningLocalized,ElfversonConvergence}.

{\rm (2)} In the case where  $h=H^3$, it is easy to see that  $\sqrt{\frac{H}{h}}=H^{-1}$. Hence \eqref{err1:FELODM} becomes
\begin{equation}\label{Oesti}
\|u_{h,h}-u^{ms,L}_{h,H}\|_{h,h}\lesssim
H\|f\|_{L^{2}(\Omega)}+H^{-1}L^{\frac{d}{2}}\theta^{L}\|f\|_{L^{2}(\Omega)},
\end{equation}
which is the same result as those of
the methods mentioned in \cite{MLocalization,petrovlod,ElfversonConvergence}.
We emphasize that in our later numerical experiments, we choose $h=H^m$ for some constant $1<m<3$, which follows that $\sqrt{\frac{H}{h}}<H^{-1}$. This means in this case the estimate \eqref{err1:FELODM} is better than \eqref{Oesti}.
\end{remark}

The following theorem gives the $L^2$ error estimate of the proposed {FE-LODM}.
\begin{theorem} \label{theol2} Suppose $\ga_0\ge\al_0$. Then we have the following estimate:
\begin{equation}\label{hh19}
\|u_{h,h}-u^{ms,L}_{h,H}\|_{L^{2}(\Omega)}\lesssim
\Big(H+\Big(\frac{H}{h}\Big)^\frac12 L^{\frac{d}{2}}\theta^{L}\Big)^{2}\|f\|_{L^{2}(\Omega)}.
\end{equation}
Moreover, there exists a positive constant $L_0$ such that when $L\ge L_0|\log (Hh)^\frac12|$, we have the following estimate,
\eq{\|u_{h,h}-u^{ms,L}_{h,H}\|_{L^{2}(\Omega)}\lesssim H^2\|f\|_{L^{2}(\Omega)}.}
\end{theorem}
\begin{proof} It suffices to prove \eqref{hh19}. Let $e_{h}=u_{h,h}-u^{ms,L}_{h,H}.$
We consider the dual problem
\begin{equation*}
 \left\{ \begin{aligned}
 -\nabla\cdot(A\nabla {w})&={e_h} \quad\mathrm{in}\,\,\Omega,\\
 w&=0\quad\,\,\mathrm{on} \,\,\partial\Omega.
 \end{aligned} \right.
 \end{equation*}
 {Let $w_{h,h}\in V_{h,h}$ be the IPCDG approximation of $w$:}
$$
a_{\Omega}(w_{h,h},\phi_{h,h})=(e_{h},\phi_{h,h})\quad \forall\,\phi_{h,h}\in V_{h,h},
$$
and let  $w^{ms,L}_{h,H}\in V^{ms,L}_{h,H}$ be the FE-LOD approximation of $w$:
$$
a_{\Omega}\big(w^{ms,L}_{h,H},v^{ms,L}_{h,H}\big)=\big(e_{h},v^{ms,L}_{h,H}\big)\quad
\forall\, v^{ms,L}_{h,H}\in V^{ms,L}_{h,H}.
$$
Further, from (\ref{solution1}) and (\ref{solution2}), it follows that
\[
a_{\Omega}\big(e_{h},w^{ms,L}_{h,H}\big)=0.
\]
Thus {from Theorem~\ref{theo}} we have
\begin{align*}
\|e_{h}\|^{2}_{L^{2}(\Omega)}&=a_{\Omega}(w_{h,h},e_{h})=a_{\Omega}\big(e_{h},w_{h,h}-w^{ms,L}_{h,H}\big)\\
&\lesssim\|e_{h}\|_{h,h}\big\|w_{h,h}-w^{ms,L}_{h,H}\big\|_{h,h}\\
&\lesssim\|e_{h}\|_{h,h}\Big(H+\Big(\frac{H}{h}\Big)^\frac12 L^{\frac{d}{2}}\theta^{L}\Big)\|e_{h}\|_{L^{2}(\Omega)},
\end{align*}
which yields
\[
\|e_{h}\|_{L^{2}(\Omega)}\lesssim\|e_{h}\|_{h,h}\Big(H+\Big(\frac{H}{h}\Big)^\frac12 L^{\frac{d}{2}}\theta^{L}\Big)
\lesssim\Big(H+\Big(\frac{H}{h}\Big)^\frac12 L^{\frac{d}{2}}\theta^{L}\Big)^{2}\|f\|_{L^{2}(\Omega)}.
\]
This completes the proof of the theorem.
\end{proof}

\section{Numerical Tests}
In this section, we first  numerically study how the size of element patches affects the errors, and then illustrate the ability of the proposed FE-LODM to deal with singularities by solving multiscale elliptic problems  with a corner singularity and high-contrast channels and steady flow transporting through highly heterogeneous porous media driven by extraction wells, respectively. For comparison, we also present results of the local orthogonal decomposition method (LODM) in \cite{MR3926249} and the combined multiscale finite element method (FE-OMsPGM) introduced in \cite{DW2016}.  We use the IPCDG solution $u_{h,h}$ to \eqref{solution1} on a very fine mesh as a reference solution. Denote the energy norm by $\|\cdot\|_E:=\|\nabla A^{\frac{1}{2}}\cdot\|_{0,\Om_1\cup\Om_2}$. We measure the relative errors of an approximate solution $U_h$ in the $L^{2}$, $L^{\infty}$ and energy norms respectively as follows:
\begin{equation*}
\frac{\|{U_h}-u_{h,h}\|_{L^2}}{\|u_{h,h}\|_{L^2}},\quad \frac{\|U_h-u_{h,h}\|_{L^{\infty}}}{\|u_{h,h}\|_{L^{\infty}}},\quad
\frac{\|U_h-u_{h,h}\|_E}{\|u_{h,h}\|_E}.
\end{equation*}

\subsection{Effect of the size of the element patches}
In this subsection we study  how the size of element patches affects the errors by simulating the following example.
\begin{example}\label{ex1}
Consider the model problem \eqref{model} on the unit square $\Omega=(0,1)\times (0,1)$ with the source term $f\equiv 1$ and  different diffusion coefficients to be specified below. And we set $\Omega_1=(\frac14,\frac38)\times(\frac14,\frac38)$ as shown in Figure~\ref{fig0}.
\end{example}

\begin{figure}
\centering
\includegraphics[width=0.5\textwidth]{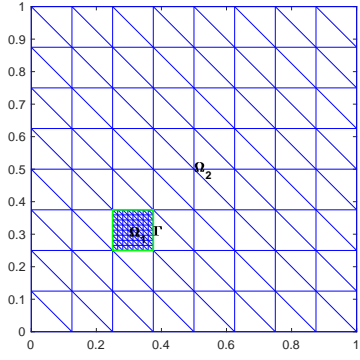}
\caption{An illustration of the separated domain used in {Example~\ref{ex1}}.}\label{fig0}
\end{figure}
First we {test the dependence of the error on the size of the element patches. Consider} the following diffusion coefficient:
\begin{equation}\label{coef1}
 {A}(x_{1},x_{2})=\frac{2+1.8\sin(2\pi x_{1}/\epsilon)}{2+1.8\cos(2\pi x_{2}/\epsilon)}+
\frac{2+1.8\sin(2\pi x_{2}/\epsilon)}{2+1.8\sin(2\pi x_{1}/\epsilon)}
\end{equation}
 with  $\epsilon=1/5$. We fix  $H$=$2^{-3}$, $h=2^{-7}$, and let the size of element patches vary.
 Table \ref{table03} shows the relative errors in the energy and $L^{2}$ norms on $\Omega$ and $\Omega_1$ between the reference solution $u_{h,h}$ and the FE-LOD solution $u_{h,H}^{ms,L}$ with $L=1, 2, 3, 6, 10$ and the ideal solution $u_{h,H}^{ms}$ as well, respectively.
\begin{table}[htp]
\caption{Example~\ref{ex1}: Relative errors for different $L$, $h=2^{-7}$, $H=2^{-3}$, $\gamma_{0}$=10.}\label{table03}
\begin{center}
\begin{tabular}{|c|c|c|c|c|} \hline
\multirow{2}{*}{\diagbox{\qquad $L$}{Error}}&
\multicolumn{2}{c|}{Error in $\Omega$}&\multicolumn{2}{c|}{Error in $\Omega_{1}$}\cr\cline{2-5}
&Energy&$L^2$&Energy&$L^2$\cr
\hline
1 & 0.1360e-00 & 0.2929e-01  & 0.2580e-01 & 0.1429e-01  \\ \hline
2  & 0.7361e-01 & 0.1127e-01 & 0.5881e-02 & 0.1371e-02   \\ \hline
3  & 0.5712e-01 & 0.8685e-02 & 0.1453e-02 & 0.1712e-03   \\ \hline
6  & 0.5534e-01 & 0.8625e-02 & 0.2586e-04 & 0.6106e-05    \\ \hline
10  & 0.5509e-01 & 0.8570e-02 & 0.5213e-06 & 0.1604e-06  \\ \hline
ideal solution  & 0.5509e-01 & 0.8569e-02 & 0.5984e-13 & 0.2713e-13 \\ \hline
\end{tabular}
\end{center}
\end{table}
It is observed that the larger the parameter $L$, the smaller the relative errors in the energy and $L^{2}$ norms on $\Omega$, and tends the errors of the ideal solution, respectively. This observation verifies the estimate in Theorem~\ref{theo}. We also notice that the errors of the FE-LOD solution on $\Om_1$ decrease very quickly {as $L$ increases.} Especially, in the ideal case, the errors on $\Om_1$ are almost equal to zero, which is coincided with the result stated in Proposition~\ref{prop1}.

 Secondly, we study how to choose the size ($L$) of the element patches to achieve the satisfied approximation behaviour for different coarse-fine grid elements. Recall that in Theorem~\ref{theo}, to balance
the error between the terms on the right-hand side of \eqref{err1:FELODM} , it is required that the localization parameter $L$ satisfies  $L\ge L_0|\log (Hh)^\frac12|$ for some positive constant $L_0$. Hence, in the following experiments, we choose $L  = \lceil L_0|\log (Hh)^\frac12| \rceil$ for different constants $L_0$. {We adopt uniform coarse meshes with sizes} $H=2^{-i}$, $i=2,3,4,5$,
and choose the fine scale reference mesh with size $h=2^{-9}$, which can resolve the multiscale feature of $A$.
The first test is done for the periodic {diffusion coefficient} defined in \eqref{coef1} with $\epsilon=1/20$, {which is denoted by $A_1$ for convenience.}

\begin{figure}[htp]
 \begin{minipage}[t]{0.5\linewidth}
  \centerline{\includegraphics[scale=0.45]{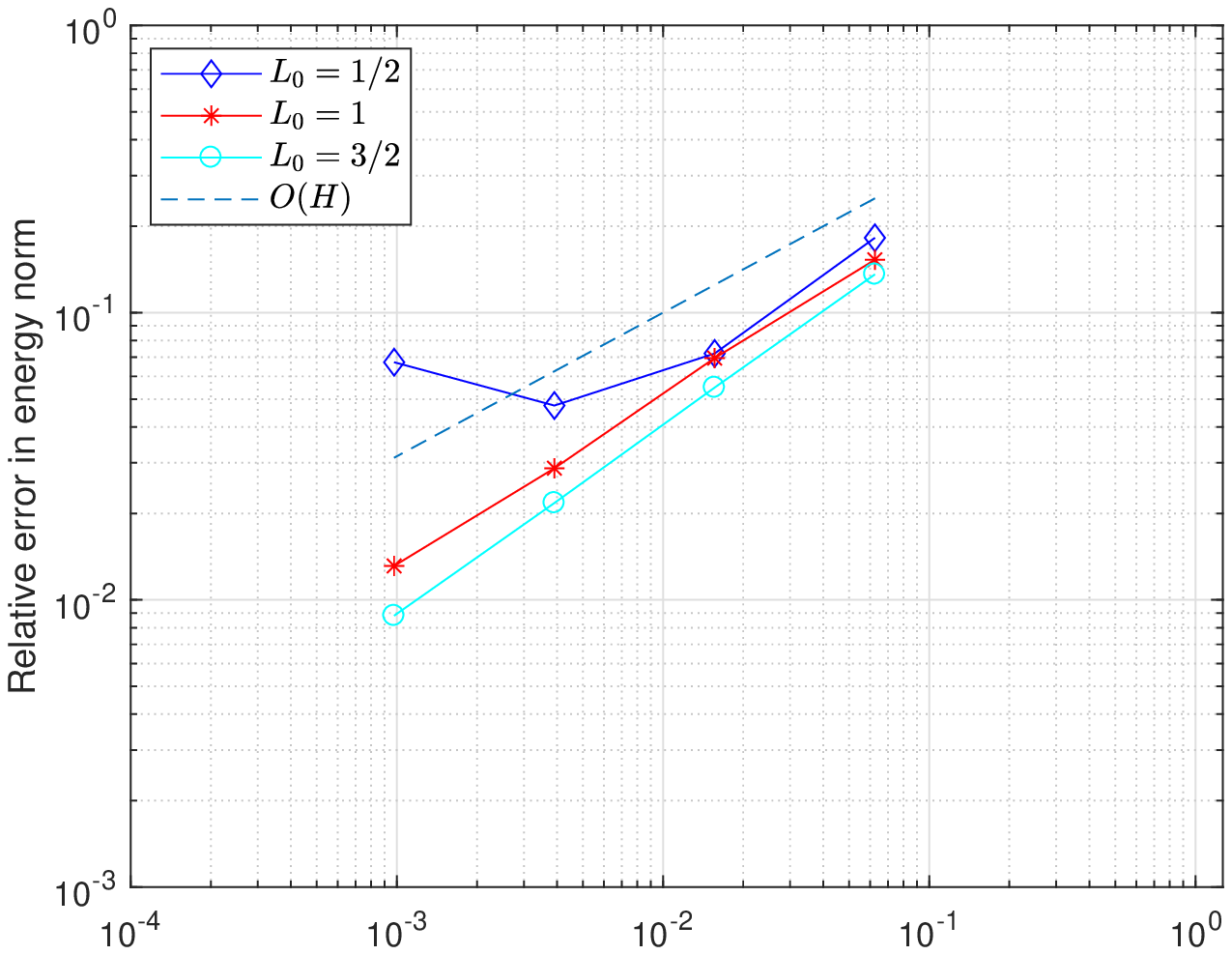}}
  \end{minipage}
  \begin{minipage}[t]{0.5\linewidth}
  \centerline{\includegraphics[scale=0.45]{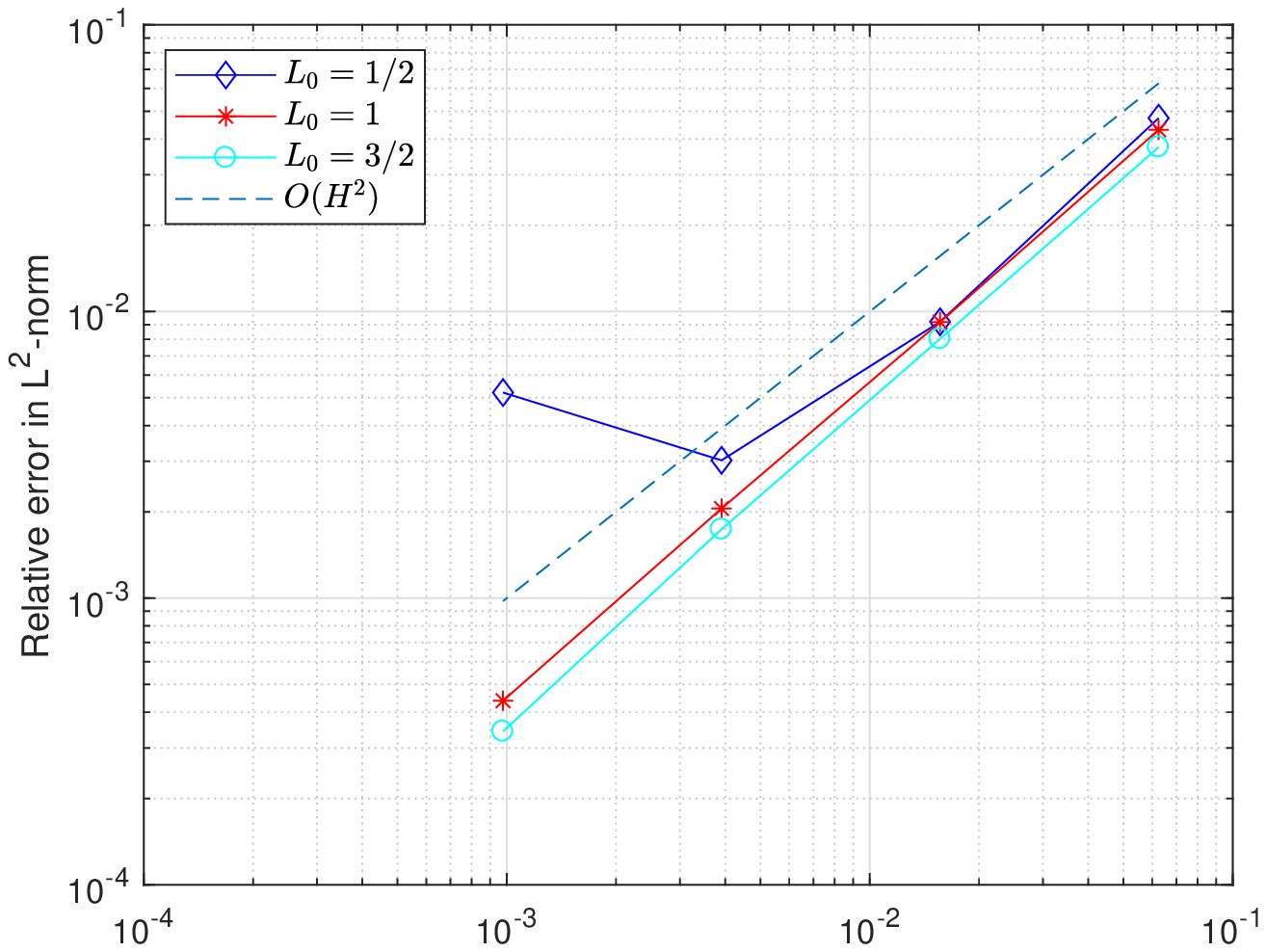}}
  \end{minipage}
\caption{{Example~\ref{ex1} with diffusion} coefficient $A_1(x)$: Relative {errors in energy-norm (left) and $L^2$-norm (right)} against the size of coarse mesh  {for $L_0=1/2,1,$ and $3/2$, respectively.} }\label{fighh}
\end{figure}

Figure~\ref{fighh} shows {the log-log plots} of the relative errors in energy-norm {(left)} and $L^2$-norm {(right)} against the size of coarse mesh ($H$) with different constants  $L_0= 1/2, 1, 3/2$,  respectively.
It is observed that {the method with $L_0=1/2$ does not performs well for large mesh size $H$, while} when $L_0$ is taken as $1$ or $3/2$, the error between the terms on the right-hand side of estimate in Theorem~\ref{theo} seems to be balanced, and the convergence rate of the energy-norm error maintains as well as that of the $L^2$-norm error.

Note that for larger localization parameter $L$,  it costs more computational effort to compute the corrector functions and cause reduced sparseness in the coarse scale stiffness matrix. Therefore in the remaining numerical experiments we use $L_0 = 1$.
In order to further illustrate the reasonability of choosing $L_0=1$, we
show the relative error results in Figure~\ref{convergfig} for {three} different diffusion coefficients $A$:
$A_1$ is defined as above; $A_2$ {is taken as the background medium in Figure~\ref{fig3}, which is a piecewise constant function on a Cartesian grid of size $2^{-9}$ and is periodic in both the
$x$- and $y$-directions};
$A_3$ is a randomly generated diffusion coefficient using the moving ellipse average technique in \cite{aver} {with} the parameters described in  Example~\ref{ex2} below.
It can be seen that {taking $L\ge |\log (Hh)^\frac12|$ (i.e. $L_0=1$) in the FE-LODM can give the optimal convergence rates in both energy- and $L^2$- norms  for all cases, which are the same as those of the ideal combined multiscale method (see Theorems~\ref{thm:ideal1}--\ref{theol2}).}

\begin{figure}[htp]
  \begin{minipage}[t]{0.5\linewidth}
  \centerline{\includegraphics[scale=0.45]{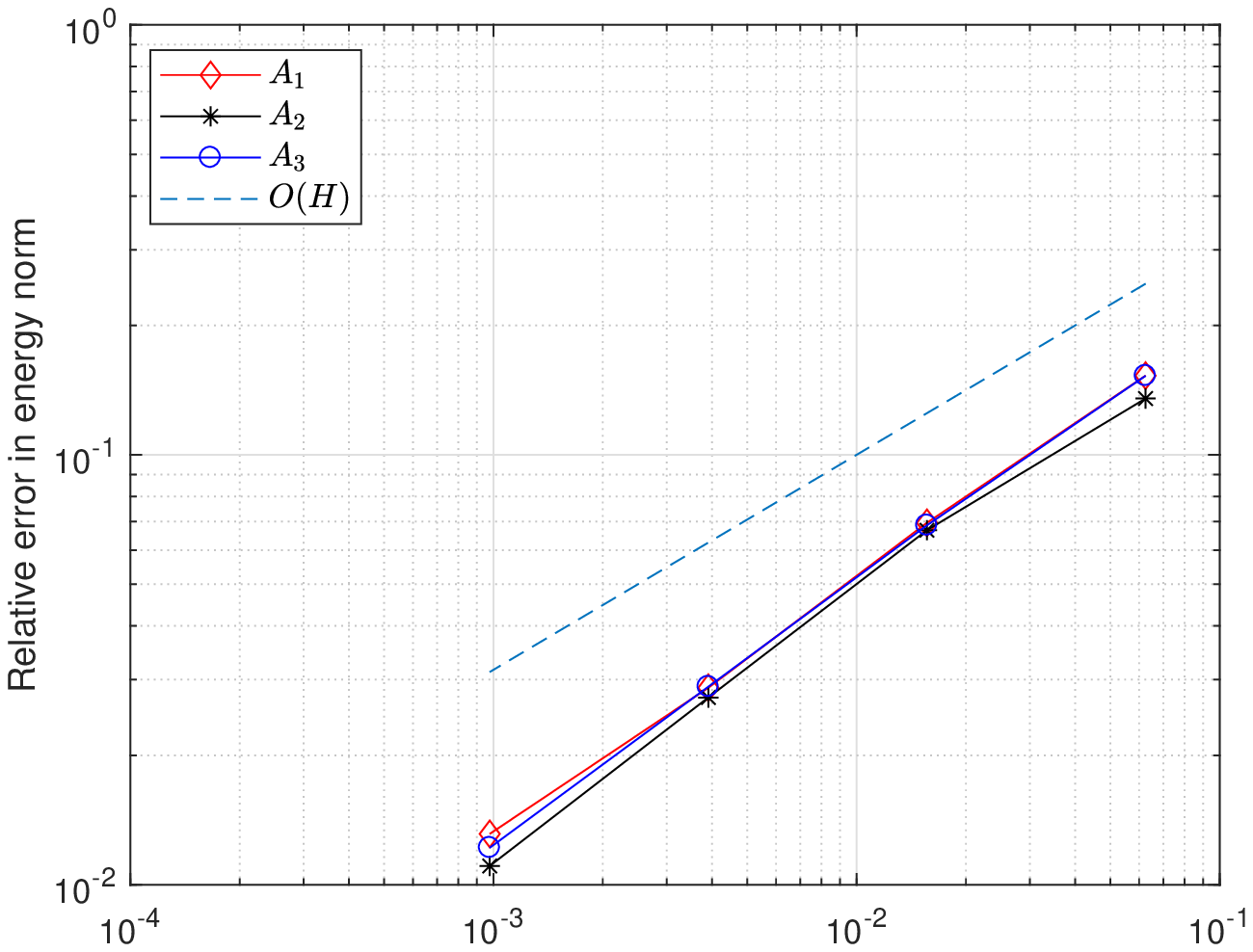}}
  \end{minipage}
  \begin{minipage}[t]{0.5\linewidth}
  \centerline{\includegraphics[scale=0.45]{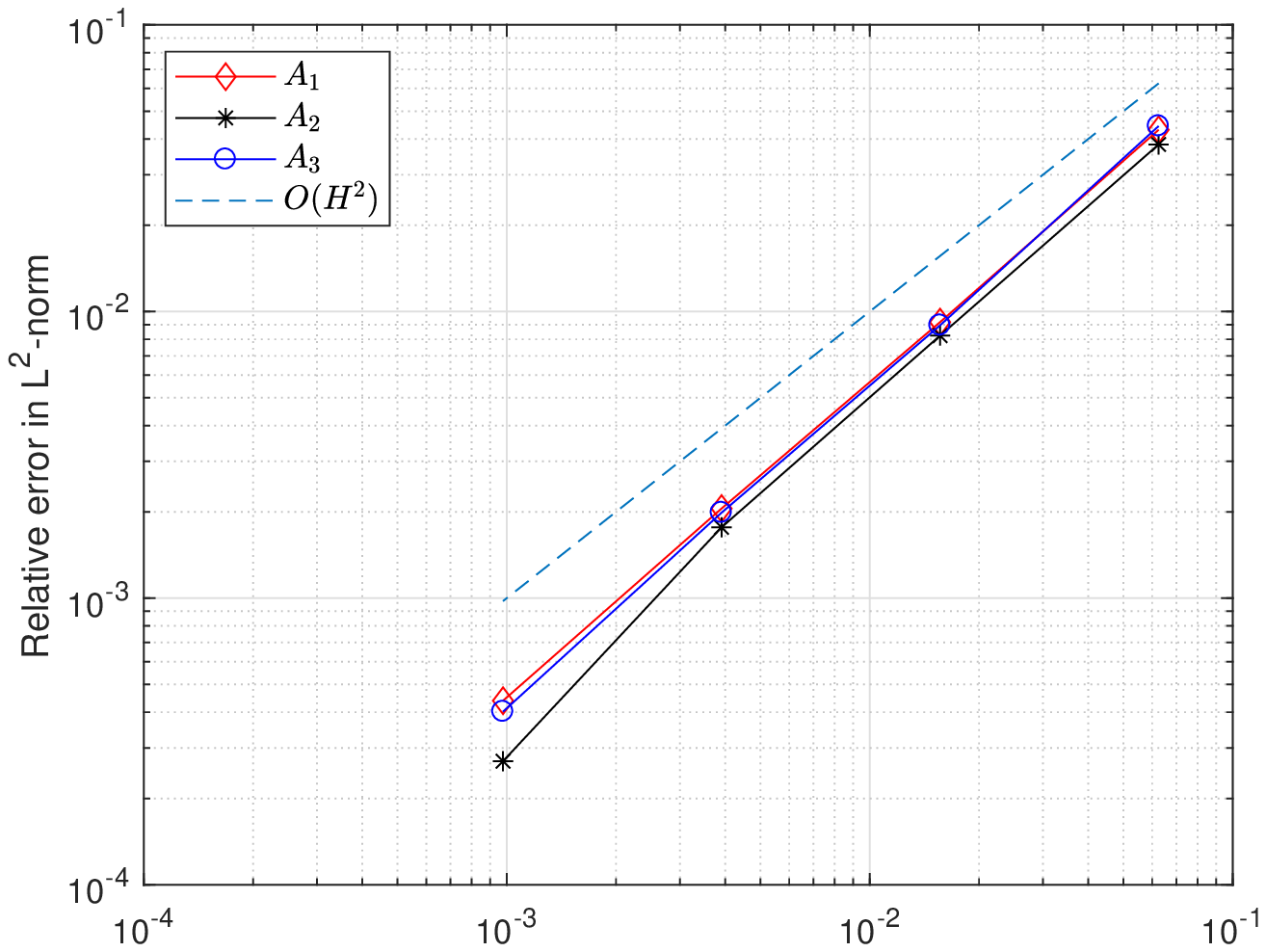}}
  \end{minipage}
\caption{{Example~\ref{ex1}}: Relative errors {in energy-norm (left) and $L^2$-norm (right)} against the size of coarse mesh for {$L_0=1$ and  $A=A_i$, $i=1,2,3,$ respectively.}}\label{convergfig}
\end{figure}

\subsection{Application to the multiscale elliptic problem with corner singularity}

In this subsection, we consider the following L-shaped domain problem
\begin{example}\label{ex2} {The multiscale problem \eqref{model} on the L-shaped domain $\Omega=\big((0,1)\times(0,1)\big)\backslash \big((\frac12,1)\times(0,\frac12)\big)$} with the random log-normal permeability field $A$, which is generated by using the moving ellipse average technique \cite{aver} with the variance of the logarithm of the permeability $\sigma^{2} = 1.5$, and the correlation lengths $l_{1} = l_{2} =0.01$ (isotropic heterogeneity) in $x_{1}$ and $x_{2}$ directions, respectively. One realization of the resulting permeability field in our numerical experiments is depicted in Figure \ref{fig2}.
\begin{figure}
\centering
\includegraphics[scale=0.56]{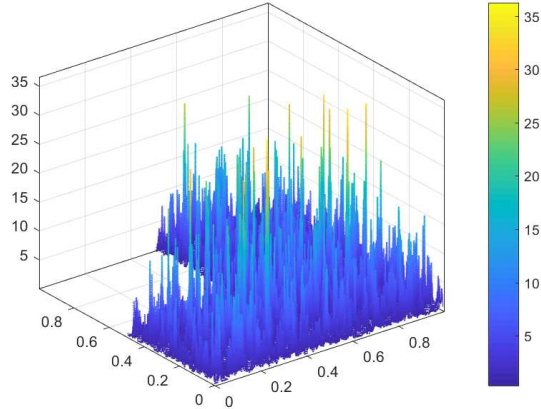}
\caption{{Example~\ref{ex2}:} The random log-normal permeability field $A$. $\frac{A_{max}}{A_{min}}$=2.9642e+03. }\label{fig2}.
\end{figure}
\end{example}
In this example, {we} set the refined subdomain $\Omega_1=\big((\frac38,\frac58)\times(\frac38,\frac58)\big)\backslash \big((\frac12,\frac58)\times(\frac38,\frac12)\big)$ to capture the singularity at the reentrant corner,  fix $H=2^{-5}$, $h=2^{-10}$, and choose the parameter $L=\lceil\log\sqrt{|Hh|}\rceil = 3$. We compare the relative errors of {the FE-LOD solution in} the $L^{2}$, $L^{\infty}$, and energy norms {with those of the LOD solution and FE-OMsPG \cite{DW2016} solution} in the whole domain as well {as in} the refined region $\Om_1$.
The errors are listed in Table \ref{tes4}. We observe that the introduced FE-LODM gives a better approximation than the LOD and FE-OMsPG methods. In particular, in the refined region $\Om_1$, our method gives much better results than the LOD method, which is very useful if one needs high-accuracy solution {in} the problematic area.
\begin{table}[htp]
\caption{{Example~\ref{ex2}:} Relative {errors} for the model problem on the L-shaped domain. $h=2^{-10}$, $H=2^{-5}$, $\gamma_{0}$=10.}\label{tes4}
\begin{center}
\begin{tabular}{|c|c|c|c|} \hline
 {{\diagbox{\qquad Method}{Error}}} & { Energy norm}
& {$L^{2}$} & {$L^{\infty}$} \\ \hline
{ LODM }       & 0.2834e-01 & 0.1553e-02 & 0.6536e-02 \\ \hline                                                           { FE-LODM} & 0.2628e-01 & 0.1449e-02 & 0.6526e-02 \\ \hline
{ FE-OMsPGM}   & 0.1045e-00 & 0.7596e-02 & 0.2424e-01 \\ \hline
{ LODM  (error in $\Omega_{1}$) }     & 0.1507e-02 & 0.1695e-02 & 0.5583e-02 \\ \hline
{ FE-LODM  (error in $\Omega_{1}$ ) }   & 0.4257e-03 & 0.4043e-03 & 0.5018e-03 \\ \hline
\end{tabular}
\end{center}
\end{table}

\subsection{Application to {the multiscale problem} with high-contrast channels}
In this subsection we use the FE-LOD method to solve the elliptic multiscale problem with high--contrast channels.
\begin{example}\label{ex3} The oscillating coefficient is set as that of \cite{DW2016}. Namely, {as shown in Figure~\ref{fig3}}, we set the high-permeability channels and inclusions with permeability values equal to $10^{5}$ and $8\times10^{4}$ respectively, and set the other values as $1$.
\begin{figure}[htp]
\centering
\includegraphics[scale=0.6]{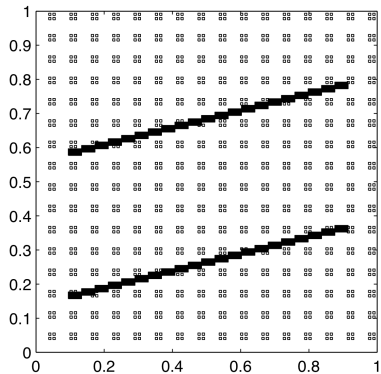}
\caption{{Example~\ref{ex3}:} Permeability field $A=10^5$ in two channels consisting of dark small rectangles; $A=8\times10^4$ in small square inclusions; $A=1$ otherwise.}\label{fig3}
\end{figure}
\end{example}
We set $\Om_1$ be the union of two layers of coarse-grid elements which contain the channels, fix $h=2^{-10}$, $H=2^{-5}$, and choose the parameter $L=\lceil\log\sqrt{|Hh|}\rceil= 3$ for this example.
The results are listed in Table \ref{table4}. It is observed that the FE-LOD method performs much better than the other two methods.
\begin{table}[htp]
\caption{{Example~\ref{ex3}:} Relative errors for the model problem with {the coefficient} given by Figure \ref{fig3}\,.\,$h=2^{-10}$\,,\,\,$H=2^{-5}$\,,\,\,$\gamma_{0}=10$.}\label{table4}
\begin{center}
\begin{tabular}{|c|c|c|c|} \hline
{ Relative error} & { Energy norm}
& $L^{2}$ & {$L^{\infty}$} \\ \hline
{ LODM } & 0.4938e-01 & 0.4882e-02 & 0.1889e-01  \\ \hline
{ FE-LODM } & 0.2169e-01 & 0.7238e-03 & 0.1248e-02\\ \hline
{ FE-OMsPGM} & 0.1063e-00 & 0.5564e-02 & 0.2580e-00 \\ \hline
\end{tabular}
\end{center}
\end{table}

\subsection{Application to {the multiscale problem with Dirac}  singularities}
In this subsection, we consider the multiscale problem with singular source terms inside the domain, which originates from the simulation of steady flow transporting through highly heterogeneous porous media driven by extraction wells. This kind of well-singularity problem is of great importance in hydrology, petroleum reservoir engineering, and soil venting techniques.

 Denote by $d({P},r)$ the disk centered at {point $P$} with radius $r >0$. {We let $\Omega$=$(0,1) \times (0,1)$ and} consider two wells $d_{j}=d({P}_{j},r), j=1,2$ with $P_{1}(\frac{1}{4},\frac{3}{4})$, $P_{2}(\frac{3}{4},\frac{1}{4})$,
 and $r=10^{-5}$. Since the size of the well (radius $r$) is negligible in situations, we make an approximation to the original single phase pressure equation by {the multiscale problem \eqref{model}  with source term $f=\sum\limits_{j=1}^{2}q_{j}\delta_{P_{j}}$ (c.f. \cite{CY2002}), where $q_{j}$ is the well flow rate  and $\delta_{P_{j}}$ is the Dirac measure at $P_{j}$. } On the well boundary $\partial d_{j}$, two quantities are of particular importance in practical applications: the well bore pressure (WBP) and the well flow rate. Here we fix the well flow rate $q_{j}$ and try to find the well bore pressure. In the computations we always take $q_{1}=-1$ and $q_{2}=1$, which corresponds to the situation that the well $d_{1}$ is an extraction well and $d_{2}$ is an injection well.

 In the following two examples, we set $\Om_1$ be the union of two small squares with edge size of $\frac{1}{16}$ centered at  points $P_{i}, i=1,2$. {And} similarly, we choose the localization parameter $L=\lceil\log\sqrt{|Hh|}\rceil=3$.

\begin{example}\label{ex4} {Let the oscillating coefficient $A$ be given by}  \begin{equation}\label{coewell}
 A(x_{1},x_{2})=\frac{1}{(2+1.5\sin\frac{2\pi x_{1}}{\epsilon})(2+1.5\sin\frac{2\pi x_{2}}{\epsilon})},
 \end{equation}
where we fix $\epsilon= 1/64$.
\end{example}
 Since the exact WBPs are unknown, we use the method introduced in \cite{CY2002} to compute them based on the well-resolved solutions obtained on a uniform $2048\times 2048$ mesh. Then we can get the ``exact" WBP $\alpha_{1}=-5.3884973$ in the first well and $\alpha_{2}=5.3884973$ in the second well (see \cite[Example~7.1]{CY2002}).
 
 In addition, we implement three other methods for comparison, including the LODM, the MsFEM introduced in \cite[Algorithm 7.1]{CY2002} (referred as G-MsFEM) and the FE-OMsPGM introduced in \cite{DW2016}. The G-MsFEM  needs to compute the discrete Green functions in a very fine mesh and it uses the developed new Peaceman method to compute the WBPs (see \cite[Section 6]{CY2002}). We also use the new Peaceman method to calculate  the WBP on each well for the LOD and FE-LOD method. For FE-OMsPGM, we use the Peaceman model  \cite{Peaceman83} to compute the WBPs since the bilinear form of FE-OMsPG method is nonsymmetric. The results are listed in Table~\ref{table5}. We can see that our FE-LODM provides a better approximation of the WBP than G-MsFEM and FE-OMsPGM, and a much better approximation than the LODM in this example.
 \begin{table}[htp]
\caption{{Example~\ref{ex4}: WBPs and relative errors at two wells.}\,$h=2^{-11}$\,,\,\,$H=2^{-6}$\,,\,\,$\gamma_{0}$=10.}\label{table5}
\begin{center}
\begin{tabular}{|c|c|c|c|c|c|} \hline
{ Methods} & { Well no.}
& { WBP} & { Relative error}\\ \hline
{ G-MsFEM} & 1 & -5.3838442 & 0.8635e-03  \\ \hline
{ FE-OMsPGM} & 1 & -5.3843102 & 0.7770e-03  \\ \hline
{ LODM } & 1 & -5.6792672 & 0.5396e-01  \\ \hline
{ FE-LODM } & 1 & -5.3876085 & 0.1649e-03 \\ \hline
{ G-MsFEM} & 2 & 5.3739254 & 0.2704e-02  \\ \hline
{ FE-OMsPGM} & 2 & 5.3843102 & 0.7770e-03  \\ \hline
{ LODM } & 2 & 5.6792672 & 0.5396e-01  \\ \hline
{ FE-LODM } & 2 & 5.3876085 & 0.1649e-03 \\ \hline
\end{tabular}
\end{center}
\end{table}

\begin{example}\label{ex5}
We generate the random permeability field $A$ on a uniform $1024\times 1024$ mesh by using the technique in \cite{aver}. Figure \ref{figeps} shows a realization of the random permeability field.

 \begin{figure}[htp]
\centering
\includegraphics[scale=0.6]{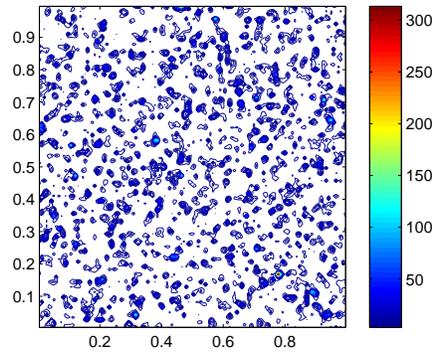}
\caption{{Example~\ref{ex5}:} The random permeability field $A$, the ratio of $\frac{A_{max}}{A_{min}}$=$6.06629e+003$.}\label{figeps}
\end{figure}
\end{example}

Using the same method as above, we can get the ``exact" well bore pressures $\alpha_{1}=-0.9860407$  and $\alpha_{2}=4.6507306$ by using the fine-grid solution on the $1024\times 1024$ mesh. The results are presented in Table \ref{table6}. We observe that {the proposed FE-LODM gives} much better approximation than the other three methods, which may be due to the fact that the FE-LODM has more accurate solution near the well than others. This superiority can also be found in the results of FE--OMsPGM, in which the local fine mesh approximation is used in the well-singularity region that is same as that of FE-LODM.

\begin{table}[htp]
\caption{{Example~\ref{ex5}: WBPs and relative errors at two wells.} $h=2^{-10}$\,,\,\,$H=2^{-6}$\,,\,\,$\gamma_{0}$=10.}\label{table6}
\begin{center}
\begin{tabular}{|c|c|c|c|c|c|} \hline
{ Methods} & { Well no.}
& { WBP} & { Relative error}\\ \hline
{ G-MsFEM} & 1 & -0.9478413 & 0.3874e-01  \\ \hline
{ FE-OMsPGM} & 1 & -0.9701813 & 0.1608e-01  \\ \hline
{ LODM } & 1 & -0.7536890 & 0.2356e-00  \\ \hline
{ FE-LODM } & 1 & -0.9864050 & 0.3695e-03 \\ \hline
{ G-MsFEM} & 2 & 1.1477417 & 0.7532e-00  \\ \hline
{ FE-OMsPGM} & 2 & 4.6391148 & 0.2498e-02  \\ \hline
{ LODM } & 2 & 2.8657275 & 0.3838e-00  \\ \hline
{ FE-LODM } & 2 & 4.6591764 & 0.1816e-02 \\ \hline
\end{tabular}
\end{center}
\end{table}

\section{Conclusions}
In this paper, we have proposed a new combined multiscale method to solve the multiscale elliptic problems which may have singularities. In order to get a good approximation of the solution in the problematic region, we use the traditional FEM directly on a very fine mesh of this subdomain, while in other region where we have a highly oscillating coefficients, we use the multiscale LODM.
The key point of implementing this idea is how to define the corrected basis function in the near interface elements. To this end, we introduce a special definition of the cell problems for the elements near the interface. The error analysis is carried out for highly varying coefficients, without any assumption on scale separation or periodicity. {Our theoretical and numerical results show} that the proposed method {is very attractive for} multiscale problems with singularities.

\appendix
\section{Proof of Lemma~\ref{inverse}}\label{a:inverse}
Given  $v_H=\sum\limits_{j\in \mathring{\mathcal{N}}_{H}}v_{j}\Phi_{j}\in V_{H,\Omega_2},$ let $w_H:=\C_Hv_H$. From \eqref{defclement} and \eqref{ChH},  it follows that
\eqn{
w_H=\sum\limits_{i\in \mathring{\mathcal{N}}_{H}}w_{i}\Phi_{i}\quad\text{with}\, w_{i}=\sum\limits_{j\in \mathring{\mathcal{N}}_{H}}\frac{(\Phi_{i},\Phi_{j})_{\Omega_{2}}}{(1,\Phi_{i})_{\Omega_{2}}}v_{j}.
}
Denote by $D=\text{diag}\big((1,\Phi_{1})_{\Omega_{2}},(1,\Phi_{2})_{\Omega_{2}},\cdots,
(1,\Phi_{N})_{\Omega_{2}}\big)$, $M=\big((\Phi_{i},\Phi_{j})_{\Omega_{2}}\big)_{N\times N}$,
 $V=(v_{i})_{N\times 1}, W=(w_{i})_{N\times 1},$
where $N$ represents the numbers of vertices in $\mathring{\mathcal{N}}_{H}$. Thus we have  $D^{-1}MV=W$, which yields $$V^{T}MV=V^{T}DW.$$
Hence, by use of the above equation, it is follows that
\eqn{
\|v_H\|^{2}_{L^{2}(\Omega)}&=V^{T}M V  \leq |V||D W|
                          \lesssim H^{-\frac{d}{2}}\|v_H\|_{L^{2}(\Omega)}H^{d}|W|\\
                              &\lesssim \|v_H\|_{L^{2}(\Omega)}\|w_H\|_{L^{2}(\Omega)},}
which yields
\eqn{
\|v_H\|_{L^{2}(\Omega)} &\lesssim \|w_H\|_{L^{2}(\Omega)}.}
Therefore $\C_H$ is an isomorphism on $V_{H,\Omega_2}$ and it holds that
\eq{\label{yinli}
\|(\C_H|_{V_{H,\Omega_2}})^{-1} w_H\|_{0,\Omega_2} &\lesssim \|w_H\|_{0,\Omega_2}.}
Thus, it follows from \eqref{kk77} that $\mathcal{C}_{h,H}$ is an isomorphism on $V_{h,H}$. For the sake of simplicity, we denote $(\C_{h,H}|_{V_{h,H}})^{-1}$ by  $\C_{h,H}^{-1}$ in the following.
From \eqref{kk77}, it is clear that $\C_{h,H}v_{h,H}-v_{h,H} \in V_{0,H}$.
Using the inverse inequality and Lemma~\ref{lem2}, we have
\eqn{
\|v_{h,H}-\C_{h,H}^{-1}v_{h,H}\|_{h,H}^{2}={}&\|\C_{h,H}^{-1}(\C_{h,H}v_{h,H}-v_{h,H})\|^{2}_{h,H}\\
={}&\|A^\frac12\nabla \big((\C_H|_{V_{H,\Omega_2}})^{-1}(\C_{h,H}v_{h,H}-v_{h,H})\big)\|^{2}_{0,\Omega_{2}}\notag\\
&+\frac{\gamma_{0}}{H}\|(\C_H|_{V_{H,\Omega_2}})^{-1}(\C_{h,H}v_{h,H}-v_{h,H})\|^{2}_{\Gamma}\\
\lesssim {}& H^{-2}\|(\C_H|_{V_{H,\Omega_2}})^{-1}(\C_{h,H}v_{h,H}-v_{h,H})\|^{2}_{0,\Omega_2.}}
Further, from \eqref{yinli} and \eqref{k2}, we have
\eqn{
\|v_{h,H}-\C_{h,H}^{-1}v_{h,H}\|_{h,H}^{2}&\lesssim H^{-2}\|\C_{h,H}v_{h,H}-v_{h,H}\|^{2}_{0,\Omega}\\
&=H^{-2}\sum\limits_{T\in {\M}_{H,\Omega_{2}}}\|\C_{h,H}v_{h,H}-v_{h,H}\|^{2}_{0,T}\\
&\lesssim \|v_{h,H}\|^{2}_{h,H}.}
Finally, using triangle inequality, we conclude that
\eqn{
\|\C_{h,H}^{-1}v_{h,H}\|_{h,H}&\leq \|\C_{h,H}^{-1}v_{h,H}-v_{h,H}\|_{h,H}+\|v_{h,H}\|_{h,H}\\
&\lesssim \|v_{h,H}\|_{h,H}.}
This completes the proof of the lemma.
\qed

\section{Proof of Lemma~\ref{lem8}}\label{a:lem8}

We first prove the inequality \eqref{lem81}. From the interpolation error estimates and the inverse inequality, we have
\eq{\label{Ihhew}\|\na I_{h,h}(\eta_T^{s,n}w)\|_{0,\Om_2}&\le \|\na I_{h,h}(\eta_T^{s,n}w)-\na (\eta_T^{s,n}w)\|_{0,\Om_2}+\|\na (\eta_T^{s,n}w)\|_{0,\Om_2}\notag\\
&\lesssim \bigg(\sum_{T\in\M_{h,\Om_2}}h_T^2|\eta_T^{s,n}w|_{2,T}^2\bigg)^\frac12+\|\na (\eta_T^{s,n}w)\|_{0,\Om_2}\notag\\
&\ls \|\na {(\eta_T^{s,n}w)}\|_{0,\Om_2}.}
 Using the triangle inequality, we obtain
\begin{align*}
\|I_{h,h}(\eta_T^{s,n}w)\|^{2}_{h,h}={}&\|A^{\frac{1}{2}}\nabla I_{h,h}(\eta_T^{s,n}w) \|^{2}_{0,\Omega_{2}}+\sum\limits_{e\in
\Gamma_{h}}\frac{\gamma_{0}}{h}\|I_{h,h}(\eta_T^{s,n}w)\|^{2}_{e}\\
\lesssim {}&\|A^{\frac{1}{2}}\nabla (\eta_T^{s,n}w) \|^{2}_{0,\Omega_{2}}+\sum\limits_{e\in \Gamma_{h}}\frac{\gamma_{0}}{h}\|\eta_T^{s,n}w-I_{h,h}(\eta_T^{s,n}w)\|^{2}_{e}\\
&+\sum\limits_{e\in \Gamma_{h}}\frac{\gamma_{0}}{h}\|\eta_T^{s,n}w\|^{2}_{e}:=\mathrm{R}_1+\mathrm{R}_2+\mathrm{R}_3.
\end{align*}
Using the fact that $\mathcal{C}_{h,H}w$=0, from \eqref{k2} and \eqref{66}, we have
\begin{align*}
\mathrm{R}_1 \lesssim &
\sum\limits_{T\in T_{n}\backslash T_{s}}\|{(w-\mathcal{C}_{h,H}w)}\nabla\eta_T^{s,n}\|^{2}_{0,T}
+\|A^{\frac{1}{2}}\nabla w\|^{2}_{0,T_{n}}\\
\lesssim &
\|H\nabla\eta_T^{s,n}\|^{2}_{L^{\infty}(\Omega)}\|\nabla w\|^{2}_{0,T_{n+1}\backslash T_{s-1}}
+\|A^{\frac{1}{2}}\nabla w\|^{2}_{0,T_{n}}\\
\lesssim & \|A^{\frac{1}{2}}\nabla w\|^{2}_{0,T_{n+1}}.
\end{align*}
Further, from Lemma~\ref{lem2}, it follows that
 \begin{align*}
\mathrm{R}_2\lesssim &
\sum_{T\in \mathcal{M}_{h,\Omega_{2}}}\frac{\gamma_{0}}{h}
\Big(h^{-1}\|\eta_T^{s,n}w-I_{h,h}(\eta_T^{s,n}w)\|^{2}_{0,T}\\
&\quad+\|\eta_T^{s,n}w-I_{h,h}(\eta_T^{s,n}w)\|_{0,T}\|\nabla (\eta_T^{s,n}w-I_{h,h}(\eta_T^{s,n}w))\|_{0,T}\Big)\\
\lesssim &\|A^{\frac{1}{2}}\nabla(\eta_T^{s,n}w) \|^{2}_{0,\Omega_{2}}=\mathrm{R}_1.
\end{align*}
For $\mathrm{R}_3$, it is easy to see
$$
\mathrm{R}_3\lesssim \sum_{\substack{e\in \Gamma_{h}, e\subset T_{n}}}\frac{\gamma_{0}}{h}\|w\|^{2}_{e}.
$$
{Combining} the above estimates of $\mathrm{R}_1, \mathrm{R}_2$ and $\mathrm{R}_3$, we have
\[
\|I_{h,h}(\eta_T^{s,n}w)\|^{2}_{h,h} \lesssim \|w\|^{2}_{h,h,T_{n+1}},
\]
which yields \eqref{lem81} immediately.

Next, we give the proof of the \eqref{lem82}. {Noting} that $w \in W_{0,h}$, and $\eta_T^{s,n}|_{T_{s}}\equiv 1$, $\eta_T^{s,n}|_{\Omega\backslash T_{n}}\equiv 0$,  it is obvious that
\[
\begin{split}
\|\eta_T^{s,n}w-I_{h,h}(\eta_T^{s,n}w)\|^{2}_{h,h}=&\|A^{\frac{1}{2}}\nabla (\eta_T^{s,n}w-I_{h,h}(\eta_T^{s,n}w)) \|^{2}_{0,T_{n}\backslash T_{s}}\\
&+\sum\limits_{e\in \Gamma_{h}}\frac{\gamma_{0}}{h}\|\eta_T^{s,n}w-I_{h,h}(\eta_T^{s,n}w)\|^{2}_{e}:=\mathrm{I}_1+\mathrm{I}_2.
\end{split}
\]
Similar to the estimate of $\mathrm{R}_2$, from Lemma \ref{lem2}, it follows that
\[
\mathrm{I}_2 \lesssim  \mathrm{I}_1.
\]
Further, {by a similar argument to \eqref{Ihhew} and} using \eqref{k2} and \eqref{66}, we have
\begin{align*}
\mathrm{I}_1&\lesssim\|A^{\frac{1}{2}}\nabla(\eta_T^{s,n}w) \|^{2}_{0,T_n\backslash T_s}\\
&\lesssim \sum\limits_{T\in T_n\backslash T_s}\|(w-\mathcal{C}_{h,H}w)\nabla\eta_T^{s,n}\|^{2}_{0,T}+
\|A^{\frac{1}{2}}\nabla w\|^{2}_{0,T_n\backslash T_s}\\
&\lesssim \|H\nabla\eta_T^{s,n}\|^{2}_{L^{\infty}(\Omega)}\|\nabla w\|^{2}_{0,T_{n+1}\backslash T_{s-1}}+\|A^{\frac{1}{2}}\nabla w\|^{2}_{0,T_n\backslash T_s}\\
&\lesssim  \|A^{\frac{1}{2}}\nabla w\|^{2}_{0,T_{n+1}\backslash T_{s-1}},
\end{align*}
which follows \eqref{lem82} immediately.
The proof of \eqref{lem83} and \eqref{lem84} is similar to the above inequality.
\qed

\section{Proof of Lemma \ref{hh12}}\label{a:hh12} The proof is divided into four steps.

\textbf{Step\,\, 1.} In this step we prove the following estimate for $L\ge 5$:
\begin{equation}\label{h10}
\|q^T_h-q_h^{T,L}\|_{h,h} \lesssim \|q^T_h\|_{h,h,\Omega\backslash T_{L-5}}.
\end{equation}

From \eqref{ahti} and \eqref{corl},  $q^T_h\in W_{0,h}$ and $q_h^{T,L}\in W_{0,h}(T_L)$ satisfy
\begin{align*}
&a_{\Om}(q^T_h,w)=a_{\tilde T}(u_{h,H},w)\quad \forall\, w\in W_{0,h},\\
&a_{\Om}(q_h^{T,L},w)=a_{\tilde T}(u_{h,H},w)\quad \forall\, w\in W_{0,h}(T_L).
\end{align*}
Subtracting the above two equations yields
\begin{equation}\label{hh11}
a_{\Om}(q^T_h-q_h^{T,L},w)=0\quad \forall\, w\in W_{0,h}(T_L).
\end{equation}
Denote {by} $e:=q^T_h-q_h^{T,L}$. Using the coercivity and continuity of $a_\Om$, from \eqref{hh11}, for any $w\in W_{0,h}(T_L)$ it follows that
\begin{align*}
\|e\|^{2}_{h,h}&\lesssim a_{\Om}(e,q^T_h-q_h^{T,L})= a_{\Om}(e,q^T_h-w)\\
& \lesssim \|e\|_{h,h}\|q^T_h-w\|_{h,h},
\end{align*}
which yields
\begin{equation}
\|q^T_h-q_h^{T,L}\|_{h,h}\lesssim {\inf_{w\in W_{0,h}(T_L)}}\|q^T_h-w\|_{h,h}.\label{h7}
\end{equation}
For {the element $T$, let $\eta_T^{L-3,L-2}$ be the cut off function defined in \eqref{bb}--\eqref{66} (with $l_1=L-3, l_2=L-2$). Since $\eta_T^{L-3,L-2}\equiv 1$ on $ T_{L-3}$ and $\eta_T^{L-3,L-2}\equiv 0 $ on $\Omega\backslash T_{L-2}$, it is easy to check that:
\eqn{ \mathcal{C}_{h,H}I_{h,h}(\eta_T^{L-3,L-2}q^T_h)= \mathcal{C}_{h,H}q^T_h=0\quad\text{on } T_{L-4},}
and hence
\eq{{\mathrm{supp}}(\mathcal{C}_{h,H}I_{h,h}(\eta_T^{L-3,L-2}q^T_h))&\subseteq T_{L-1}\backslash T_{L-4},\label{sup}\\
{\mathrm{supp}}(\mathcal{C}_{h,H}^2I_{h,h}(\eta_T^{L-3,L-2}q^T_h))&\subseteq T_{L}\backslash T_{L-5}.\label{sup1}}}
Using Lemma \ref{cc}, for $\mathcal{C}_{h,H}I_{h,h}(\eta_T^{L-3,L-2}q^T_h)$, there is a $\mu\in V_{0,h}$
such that
\eqn{\mathcal{C}_{h,H}\mu=\mathcal{C}_{h,H}I_{h,h}(\eta_T^{L-3,L-2}q^T_h),
\quad &{\|\mu\|_{h,h}\ls\|\mathcal{C}_{h,H}I_{h,h}(\eta_T^{L-3,L-2}q^T_h)\|_{h,H},}\\
&\text{and} \quad {\mathrm{supp}}(\mu)\subseteq T_{L}\backslash T_{L-5}.}
Further, using  {\eqref{sup}, Lemmas~\ref{lem} and} \ref{lem8}, we have
\begin{equation}\label{h9}
\begin{split}
\|\mu\|_{h,h}&\lesssim \|\mathcal{C}_{h,H}I_{h,h}(\eta_T^{L-3,L-2}q^T_h)\|_{h,H,T_{L-1}\backslash T_{L-4}}\\
&\lesssim \|I_{h,h}(\eta_T^{L-3,L-2}q^T_h)\|_{h,h,T_{L}\backslash T_{L-5}}\\
&= \|I_{h,h}(\eta_T^{L-3,L-2}q^T_h)\|_{h,h,T_{{L-2}}\backslash T_{L-3}}+\|q^T_h\|_{h,h,T_{L-3}\backslash T_{L-5}}\\
&\lesssim \|q^T_h\|_{h,h,T_{L-1}\backslash T_{L-5}}.
\end{split}
\end{equation}

Hence taking $w=I_{h,h}(\eta_T^{L-3,L-2}q^T_h)-\mu \in W_{0,h}(T_L)$ in \eqref{h7} {and using} \eqref{h9} and Lemma \ref{lem8}, we have
\begin{align*}
\|q^T_h-q_h^{T,L}\|_{h,h} &\lesssim \|I_{h,h}(1-\eta_T^{L-3,L-2})q^T_h\|_{h,h}+\|\mu\|_{h,h}\\
 &\lesssim \|q^T_h\|_{h,h,\Omega\backslash T_{L-4}}+
\|q^T_h\|_{h,h,T_{L-1}\backslash T_{L-5}},
\end{align*}
which implies that \eqref{h10} holds.

\textbf{Step\,\, 2}.  Suppose we can prove  the following recursive inequality
\begin{equation}\label{h11}
\|q^T_h\|_{h,h,\Omega\backslash T_{M}}\le  \theta_0\|q^T_h\|_{h,h,\Omega\backslash T_{m}}\quad\forall\, m=M-5\ge 0,
\end{equation}
where $0<\theta_0<1$ is a constant independent of $M$ and $q^T_h$.

For $L=5k+j$ with integers $k\ge 1$ and $0\le j\le 4$,
setting $\theta=\theta_0^\frac15$ and using \eqref{h11}  repeatedly, we conclude that
\begin{align}
\|q^T_h\|_{h,h,\Omega\backslash T_{L-5}}&\lesssim \theta_0^{k-1}\|q^T_h\|_{h,h,\Omega\backslash T_j}\lesssim \theta^{L-j-5}\|q^T_h\|_{h,h}\nn\\
&\lesssim \theta^L\|q^T_h\|_{h,h}.\label{h122}
\end{align}
Clearly, the above estimate also holds for $5\le L\le 9$ and hence \eqref{h122} holds for $L\ge 5$.

\textbf{Step\,\, 3}. In this step we prove  \eqref{h11}. Let $\varepsilon=1-\eta_T^{m+2,M-2}$ {satisfying} $\varepsilon\equiv 1$ in $\Omega\backslash T_{M-2}$, $\varepsilon\equiv 0$ in $T_{m+2}$, and $0<\varepsilon<1$ otherwise. It is easy to see that
\begin{equation}\label{eqa11}
\begin{split}
\|q^T_h\|^{2}_{h,h,\Omega\backslash T_{M}}
& \leq \|\varepsilon q^T_h\|^{2}_{h,h}\lesssim a_{\Om}(\varepsilon q^T_h,\varepsilon q^T_h)\\
&=a_{\Om}(q^T_h,\varepsilon^{2} q^T_h)
+(A{\nabla_h} \varepsilon \cdot q^T_h,\nabla_h \varepsilon\cdot q^T_h).
\end{split}
\end{equation}
For the function $\mathcal{C}_{h,H}I_{h,h}(\varepsilon^{2}q^T_h)$, using Lemma \ref{cc}, there exists a $\gamma\in V_{0,h}$ such that $\mathcal{C}_{h,H}\gamma=\mathcal{C}_{h,H}I_{h,h}(\varepsilon^{2}q^T_h)$, and
\eqn{
{\mathrm{supp}}(\mathcal{C}_{h,H}I_{h,h}(\varepsilon^{2}q^T_h))\subseteq T_{M-1}\backslash T_{m+1},\\
{\mathrm{supp}}(\gamma)\subseteq {\mathrm{supp}}(\mathcal{C}_{h,H}^2I_{h,h}(\varepsilon^{2}q^T_h))\subseteq T_{M}\backslash T_{m}.
}
In addition, it holds that
\begin{equation}\label{fy}
\|\gamma\|_{h,h}\lesssim \|\mathcal{C}_{h,H}I_{h,h}(\varepsilon^{2}q^T_h)\|_{h,H}.
\end{equation}
Since $I_{h,h}(\varepsilon^{2}q^T_h)-\gamma \in W_{0,h}(\Om_2\backslash T_{m})$,
from \eqref{ahti}, it follows that
\[
a_{\Om}(q^T_h,I_{h,h}(\varepsilon^{2}q^T_h)-\gamma)=a_{\tilde T}(u_{h,H},I_{h,h}
(\varepsilon^{2}q^T_h)-\gamma)=0,
\]
which combines \eqref{eqa11} yields
\begin{align*}
\|q^T_h\|^{2}_{h,h,\Omega\backslash T_{M}}&\lesssim a_{\Om}(q^T_h,\varepsilon^{2} q^T_h-I_{h,h}(\varepsilon^{2} q^T_h))
+a_{\Om}(q^T_h,\gamma)\\
&\quad +(A\nabla \varepsilon \cdot q^T_h,\nabla \varepsilon\cdot q^T_h):=\mathrm{T}_1+\mathrm{T}_2+\mathrm{T}_3.
\end{align*}
Using the same argument as that of \eqref{lem82}, we obtain
\begin{align*}
\mathrm{T}_1
&\lesssim\|q^T_h\|_{h,h,T_{M-2}\backslash T_{m+2}}\|\varepsilon^{2} q^T_h-I_{h,h}(\varepsilon^{2}q^T_h)\|_{h,h,T_{M-2}\backslash T_{m+2}}\\
& \lesssim \|q^T_h\|^{2}_{h,h,T_{M-1}\backslash T_{m+1}}.
\end{align*}
Further, for $\mathrm{T}_2$, using the same argument as that of \eqref{lem81}, from \eqref{fy}, we have
\begin{align*}
\mathrm{T}_2&\lesssim\|q^T_h\|_{h,h,T_{M}\backslash T_{m}}\|\gamma\|_{h,h}\\
&\lesssim\|q^T_h\|_{h,h,T_{M}\backslash T_{m}}\|\mathcal{C}_{h,H}I_{h,h}(\varepsilon^{2}q^T_h)\|_{h,H,T_{M-1}\backslash T_{m+1}}\\
&\lesssim\|q^T_h\|_{h,h,T_{M}\backslash T_{m}}\|I_{h,h}(\varepsilon^{2}q^T_h)\|_{h,h,T_{M}\backslash T_{m}}\\
&=\|q^T_h\|_{h,h,T_{M}\backslash T_{m}}\left(\|q^T_h\|_{h,h,T_{M}\backslash T_{M-2}}+\|I_{h,h}(\varepsilon^{2}q^T_h)\|_{h,h,T_{M-2}\backslash T_{m+2}}\right)\\
&\lesssim \|q^T_h\|^{2}_{h,h,T_{M}\backslash T_{m}}.
\end{align*}
To estimate $\mathrm{T}_3$, by use of the assumption \eqref{bb}--\eqref{66} and \eqref{k2}, it follows that
\begin{align*}
\mathrm{T}_3 &=\|A^{\frac{1}{2}}\nabla\varepsilon\cdot q^T_h\|^{2}_{0,\Omega}=\sum\limits_{T\in \mathcal{M}_{H,\Omega_{2}}}\|A^{\frac{1}{2}}\nabla\varepsilon\cdot q^T_h\|^{2}_{0,T}\\
&\lesssim \sum\limits_{T\in T_{M-2}\backslash T_{m+2}}\|\nabla \varepsilon\|_{L^{\infty}(\Omega)}^{2}\|q^T_h-\mathcal{C}_{H}q^T_h\|^{2}_{0,T}\\
&\lesssim \|q^T_h\|^{2}_{h,h,T_{M-1}\backslash T_{m+1}}.
\end{align*}
Thus, by using the above estimates of $\mathrm{T}_1, \mathrm{T}_2, $ and $\mathrm{T}_3$, we have, for some positive constant $C_0$,
\begin{align*}
\|q^T_h\|^{2}_{h,h,\Omega\backslash T_{M}}&\le C_0\|q^T_h\|^{2}_{h,h,T_{M}\backslash T_{m}}=C_0 \|q^T_h\|^{2}_{h,h,\Omega\backslash T_{m}}-C_0\|q^T_h\|^{2}_{h,h,\Omega\backslash T_{M}},
\end{align*}
which implies that \eqref{h11} holds with $\theta_0:=\big(\frac{C_0}{C_0+1}\big)^\frac12.$

\textbf{Step\,\, 4}. Next, we estimate $\|q^T_h\|_{h,h}^{2}$ in \eqref{h122}.
\begin{align*}
\|q^T_h\|_{h,h}^{2}&\lesssim a_{\Om}(q^T_h,q^T_h)=a_{\tilde T}(u_{h,H},q^T_h)\lesssim \|u_{h,H}\|_{h,h,\tilde{T}}\|q^T_h\|_{h,h},
\end{align*}
where $\tilde{T}$ is defined in Section \ref{sec34},
which together with \eqref{h10} and \eqref{h122} completes the proof of the lemma.
\qed

\section*{References}
\bibliography{ref}
\end{document}